\documentclass{amsart}

\usepackage{amssymb}
\usepackage{amsthm}
\usepackage{amsmath}

\usepackage{mathrsfs}

\usepackage{tikz}
\usetikzlibrary{angles,quotes}
\usetikzlibrary{fit, calc}
\usetikzlibrary{decorations.pathreplacing}

\usepackage{hyperref}

\theoremstyle{plain}
\newtheorem{proposition}{Proposition}[section]
\newtheorem{theorem}[proposition]{Theorem}
\newtheorem{lemma}[proposition]{Lemma}

\theoremstyle{definition}
\newtheorem{example}[proposition]{Example}
\newtheorem{definition}[proposition]{Definition}
\newtheorem{observation}[proposition]{Observation}
\theoremstyle{remark}
\newtheorem{remark}[proposition]{Remark}

\DeclareMathOperator{\vol}{vol}
\DeclareMathOperator{\diam}{diam}

\DeclareMathOperator{\dist}{d}
\DeclareMathOperator{\CAT}{\mathsf{CAT}}
\DeclareMathOperator{\Isom}{\mathsf{Isom}}
\DeclareMathOperator{\Ac}{\mathcal{A}}
\DeclareMathOperator{\Bc}{\mathcal{B}}
\DeclareMathOperator{\Cc}{\mathscr{C}}
\DeclareMathOperator{\Fc}{\mathcal{F}}

\DeclareMathOperator{\Nc}{\mathcal{N}}
\DeclareMathOperator{\Oc}{\mathcal{O}}

\DeclareMathOperator{\Sc}{\mathcal{S}}

\DeclareMathOperator{\Nb}{\mathbb{N}}

\DeclareMathOperator{\Rb}{\mathbb{R}}

\newcommand{\abs}[1]{\left|#1\right|}

\newcounter{countharry}
\newcounter{countandy}
\date{\today}

\begin{document}

\title{On the almost sure spiraling of geodesics in $\CAT(0)$ spaces}
\author{Harrison Bray} 
\address{George Mason University}
\email{hbray@gmu.edu}
\author{Andrew Zimmer}
\address{University of Wisconsin-Madison}
\email{amzimmer2@wisc.edu}
\date{\today}
\subjclass[2020]{}

\begin{abstract} We prove a logarithm law-type result for the spiraling of geodesics around certain types of compact subsets (e.g. quotients of periodic  Morse flats) in quotients of 
rank one $\CAT(0)$ spaces. 
\end{abstract}

\maketitle

\section{Introduction}

For a finite volume noncompact hyperbolic manifold $M$, 
ergodicity of the geodesic flow 
implies that a typical geodesic is dense in $M$.  
Hence, such a geodesic enters and exits
a fixed cusp infinitely often.
Sullivan's  logarithm law states that the 
maximal excursion of a generic geodesic in the neighborhood of a cusp is
logarithmic in the time traveled \cite{sullivan82}.

  In Sullivan's case, the fundamental group is relatively hyperbolic with respect to the fundamental groups of the cusps. One of the main motivating examples of this work is another classical example of a manifold with relatively hyperbolic fundamental group, which arises from the geometrization theorem for 3-manifolds. 

Let $M$ be an oriented closed non-geometric irreducible 3-manifold
obtained by gluing hyperbolic pieces along tori. Then $\pi_1(M)$ is
relatively hyperbolic with respect to the fundamental groups of the gluing
tori~\cite{Dahmani2003}. Furthermore, by a result of Leeb~\cite{Leeb1995}, we can endow
$M$ with a non-positively curved Riemannian metric and, using the flat torus theorem~\cite{GW1971,LY1972}, we can assume that the tori in the geometric decomposition are flat and totally geodesic. 

The universal cover of $M$ has rank one axes, and so by a theorem of Knieper~\cite{Knieper1998} there exists a unique measure of maximal entropy $\mathsf{m}$ for the geodesic flow on the unit tangent bundle $T^1 M$ of $M$.  Given $v \in T^1 M$, let $\ell_v : [0,\infty) \rightarrow M$ denote the geodesic ray with $\ell_v'(0)=v$. 

Fix a gluing torus $\mathbb{T} \subset M$  and let $\hat{\mathbb{T}}$ be the
union of all parallel tori (which is isometric to $\mathbb{T}$ times a
closed interval). Then given $v \in T^1 M$, $\epsilon>0$, and $t\geq 0$, we
define the {\em penetration} $\mathfrak p_{\epsilon}(v,t)$ to be the maximal
length of an interval $I$ containing $t$ so that its image 
$\ell_v(I)$ is  contained in an $\epsilon$-neighborhood of $\hat{\mathbb{T}}$, and equals 0
if no such interval exists. 

In this setting we have the following logarithm law,  which can be viewed as a flat torus version of Sullivan's cusp excursion theorem. 

\begin{theorem}[see Section~\ref{sec:proof of thm nonpositive}]
  Let $M$ and $\mathbb{T} \subset M$ be as above. Then 
  for any $\epsilon>0$ sufficiently small and $\mathsf{m}$-a.e. $v \in T^1 M$, 
    \[
    \limsup_{t\to+\infty} \frac{\mathfrak p_{\epsilon}(v,t)}
    {\log t}=\frac1{h_{top}(M)}, 
  \]
  where $h_{top}(M)$ is the topological entropy of the geodesic flow on $T^1M$.
  \label{thm:nonpositive}
\end{theorem}

The above Theorem~\ref{thm:nonpositive} 
is a special case of our main result, Theorem~\ref{thm:main} below, which also generalizes the work
of Hersonsky--Paulin from the setting of $\CAT(-1)$ metric spaces to the setting of rank one $\CAT(0)$
metric spaces (which are also called rank one {\em Hadamard spaces} in the literature). In particular, Hersonsky--Paulin proved a version of
Theorem~\ref{thm:main} for neighborhoods of compact totally geodesic
embedded subspaces of a hyperbolic manifold \cite{HP2010}, or more
generally quotients of $\CAT(-1)$ metric spaces.

In this work, we identify key geometric assumptions that allow us to 
extend their work to $\CAT(0)$ metric spaces.
In recent years, a number of authors have studied the geodesic flow on such spaces \cite{Knieper1997,Knieper1998, ricks,
LinkPicaud2016,Link2018}. The geodesic flow is no longer (metric) Anosov as
in the $\CAT(-1)$ setting,
but still has dynamics resembling (nonuniform) hyperbolicity.

We now present the background for our main result, Theorem~\ref{thm:main},
which we state at the level of the universal cover. For precise definitions, see Section~\ref{s:defns}.

Suppose $\Gamma$ is a discrete group of isometries  of
a $\CAT(0)$-metric space $X$. Given a  convex subset $\Cc$ of $X$, a unit speed geodesic ray $\ell \colon
[0,\infty) \rightarrow X$
and $\epsilon > 0$, we define the {\em penetration} $\mathfrak p_{\Cc,\epsilon}(\ell,t)$ to be the maximal
length of an interval $I$ containing $t$ so that its image 
$\ell_v(I)$ is  contained in an $\epsilon$-neighborhood of $\alpha\Cc$ for some $\alpha\in\Gamma$,  and equals 0
if no such interval exists.

We will assume that the convex set $\Cc$ is {\em Morse}. This well-studied
property intuitively means in this setting that 
geodesic triangles transverse to $\Cc$ are thin
(see Proposition~\ref{thm:equiv
of Morse} for a precise statement). 
We also assume $\Cc$ admits a cocompact action by an {\em almost
malnormal subgroup}, which most crucially implies the diameter of the
intersection of $\Cc$ with any of its $\Gamma$-translates is uniformly
bounded (see Proposition~\ref{prop:diam_flat_intersection}).

We establish the following logarithm law in the $\CAT(0)$ setting.   Recall the {\em critical exponent} of a group $\Gamma$ acting
  on a metric $X$ by isometries is 
  \[
    \delta(\Gamma):=\limsup_{t\to\infty} \frac1t\log\#\{\gamma\in\Gamma
    \mid d(o,\gamma o)\leq t\},
  \]
where $o \in X$ is any fixed point. See Subsection~\ref{s:notation} for additional definitions.

\begin{theorem}[see Section~\ref{s:loglaw}]\label{thm:main} 
  Let $\Gamma$ be a rank one discrete group of isometries acting 
  on a proper $\CAT(0)$ metric space $X$ with finite critical exponent $\delta(\Gamma)$.
  Assume $\Cc\subset X$ is convex and
  that $\Gamma_0<\Gamma$ acts cocompactly on $\Cc$. Moreover, assume
\begin{enumerate}
\item\label{item:Morse assumption} $\Cc$ is a Morse subset.
\item\label{item:parallel line assumption} $\Cc$ contains all geodesic lines in $X$ which are parallel to a geodesic line in $\Cc$. 
\item\label{item:malnormal assumption} $\Gamma_0$ is almost malnormal and has infinite index in $\Gamma$.
\item\label{item:Gamma0 counting assumption}     For any $o \in X$, there exists a       positive-valued    polynomial $Q \colon \Rb \rightarrow (0,\infty)$, $C_1 > 1$, and $n_0>0$ such that 
$$
C_1^{-1} Q(n)e^{\delta(\Gamma_0)n} \leq \#\{ \gamma \in \Gamma_0 : n \leq \dist(o,\gamma o) \leq n + n_0\} \leq C_1 Q(n)e^{\delta(\Gamma_0)n}
$$
for all $n \geq 1$. 

\item 
  \label{item:Gamma_counting_assumption}      For any $o \in X$, there exists $C_2 > 1$ such that 
  $$
 C_2^{-1} e^{\delta(\Gamma) n} \leq \#\{\gamma \in \Gamma : \dist(o,\gamma o) \leq n\}\leq C_2 e^{\delta(\Gamma) n}
  $$
  for all $n \geq 1$. 
\end{enumerate}
If $\mu$ is a Patterson--Sullivan measure for $\Gamma$ with dimension
$\delta(\Gamma)$, then  for $\mu$-almost every $\xi \in \partial X$ and
every unit speed geodesic ray $\ell \colon [0,\infty) \rightarrow X$ limiting to $\xi$, we have 
\begin{equation}\label{eqn:log law in main theorem}
\limsup_{t \rightarrow \infty} \frac{\mathfrak p_{ \Cc,\epsilon}(\ell, t)}{\log t} = \frac{1}{\delta(\Gamma) - \delta(\Gamma_0)}.
\end{equation}

\end{theorem} 

\begin{remark} In Proposition~\ref{prop:entropy_gap}, we confirm that $\delta(\Gamma)>\delta(\Gamma_0)$ under the assumptions of Theorem~\ref{thm:main} and so the right hand side of Equation~\eqref{eqn:log law in main theorem} is a finite number. 
\end{remark} 

\begin{remark}
  Note that Ricks generalized Knieper's result on the existence and
  uniqueness of the measure of maximal entropy for compact, geodesically
  complete, locally $\CAT(0)$ spaces
  \cite{ricks, ricks_uniqueness, Knieper1998}.
  As for rank one Riemmanian manifolds,
  the critical exponent agrees with the topological entropy of the geodesic
  flow   and the measure of maximal entropy can be constructed using Patterson--Sullivan measures.

  \label{rem:on_ricks}
\end{remark}

Theorem~\ref{thm:nonpositive} arises from Theorem~\ref{thm:main} by taking
$\Cc$ to be a lift of $\hat{\mathbb{T}}$, and $\Gamma_0$ the stabilizer in
$\Gamma$ of $\Cc$. 
More generally, when $F$ is a periodic Morse flat, then a ``thickening'' of $F$, as in Proposition~\ref{prop:Morse flats}, satisfies Theorem~\ref{thm:main}. In particular, in cocompact $\CAT(0)$-spaces with isolated flats (in the
sense of Hruska--Kleiner~\cite{HK2005}) every maximal flat  (which is a flat
that is not contained in any higher dimensional flat) has a thickening which satisfies Theorem~\ref{thm:main}. Notice that the universal cover of the manifold $M$ appearing in Theorem~\ref{thm:nonpositive} has isolated flats and the lift of one of the gluing tori is a maximal flat. See Section~\ref{sec:Morse Flats} for more details.

\subsection{The assumptions of Theorem~\ref{thm:main}}
\label{sec:assumptions}

Hersonsky--Paulin
established Theorem~\ref{thm:main} for $\CAT(-1)$ spaces
\cite[Theorem 5.6]{HP2010}. The Morse \eqref{item:Morse assumption}, 
parallel
\eqref{item:parallel line assumption}, and counting
\eqref{item:Gamma0 counting assumption} assumptions 
are not explicit in Hersonsky--Paulin's result, but are consequences of $X$ being
$\CAT(-1)$.   The almost malnormal property~\eqref{item:malnormal assumption}
  and the exponential growth assumption on $\Gamma$
\eqref{item:Gamma_counting_assumption} are assumed explicitly in
Hersonsky--Paulin's work. Assumption~\eqref{item:Gamma_counting_assumption} 
is known in the $\CAT(-1)$ setting for many cases of interest,
including the cocompact case by work of Coornaert \cite[Th\'eor\`eme
7.2]{Coornaert}.  More generally, 
  when $\Gamma$ is a rank one discrete group of isometries acting on a proper
  $\CAT(0)$ space, 
  \eqref{item:Gamma_counting_assumption} is known 
  if $\Gamma$ is divergent type with nonarithmetic length spectrum
  and finite measure of maximal entropy \cite[Theorem C]{Link2020}.
  This applies for instance to the setting of
  Theorem~\ref{thm:nonpositive} (see Section~\ref{sec:proof of thm
  nonpositive}). The almost malnormal property~\eqref{item:malnormal assumption} seems essential for the proofs in both cases.

For $\CAT(-1)$ spaces, the Morse property \eqref{item:Morse assumption}
follows from Gromov hyperbolicity, and the  parallel property
\eqref{item:parallel line assumption} is a consequence of the fact that parallel geodesics have the same image. Since $\Gamma_0$ acts cocompactly on $\Cc$, counting property \eqref{item:Gamma0 counting assumption} again follows from 
Coornaert \cite{Coornaert}. In particular, in the $\CAT(-1)$-setting  
the growth of the orbital counting function 
for a convex cocompact subgroup is purely exponential.

Note that weakening the growth function to allow for some polynomial growth is important in the $\CAT(0)$ setting; for instance the  orbital counting function for a
periodic $d$-flat grows like $n^d$. In fact, this counting assumption~\eqref{item:Gamma0 counting assumption}
is always satisfied when $\Cc$ is a non-positively curved Riemannian
manifold~\cite{Knieper1997}, and it is possible it is true in general, but verifying this would probably require characterizing the higher rank $\CAT(0)$ spaces admitting geometric actions.

Given a Morse subset in a $\CAT(0)$ space, it is always possible to thicken it to satisfy
Assumption~\eqref{item:parallel line assumption}, see
Proposition~\ref{prop:Cond 2 holds after thickening} below. Furthermore, this assumption is designed to avoid examples of the following form. 

\begin{example} Suppose $X$, $\Gamma$, $\Cc$, and $\Gamma_0$ satisfy
  Theorem~\ref{thm:main}. Then consider $X' : = X \times [0,1]$ with the
  product metric  and $\Cc' : = \Cc \times \{0\}$. The $\Gamma$ action on
  $X$ extends to a $\Gamma$ action on $X'$ by acting trivially in the
  second factor. Then $X$, $\Gamma$, $\Cc'$, and $\Gamma_0$ satisfy all the
  assumptions in Theorem~\ref{thm:main} except for ~\eqref{item:parallel
  line assumption}. Furthermore, if $\ell : [0,\infty) \rightarrow X$ is a geodesic and $r \in [0,1]$, then  $\tilde \ell(t) = (\ell(t), r)$ is a geodesic in $X'$ whose projection does not intersect 
  an $\epsilon$-neighborhood of $\Gamma_0 \backslash \Cc'$
  for any $r > \epsilon$.
\end{example}

\subsection{Khinchin-type theorem}
\label{s:khinchin}

As in Sullivan's original work, the
logarithm law arises as an application of a Khinchin-type theorem. 
This result is inspired by Khinchin's strong 0-1 law for the real line: 
for a
function $\psi\colon \mathbb N\to\mathbb R^+$ which 
is sufficiently regular, the limsup set 

\begin{equation}
  \Theta(\psi):=\left\{x\in [0,1]: \abs{x-\frac{p}q}<\frac{\psi(q)}q
  \textrm{ for infinitely many reduced }\frac{p}q\in\mathbb Q\right\}
  \label{eqn:khinchin_log_law}
\end{equation}
has
probability 1 if $\psi$ is summable, and probability 0 otherwise. In
particular, taking the family $\psi(q)=q^{-1}$,
we recover 
a probability 1 version of Dirichlet's theorem.

  Although Khinchin's original theorem does not apply verbatim  
  to our setting, Sullivan's arguments have direct parallels. Sullivan
  first proved a
  Khinchin-type theorem analogous to \eqref{eqn:khinchin_log_law}. 
  He then proved the logarithm law as an
  application of this Khinchin-type theorem, analogous to the probability 1
  version of Dirichlet's theorem.

As in Sullivan, our logarithm law (Theorem~\ref{thm:main}) will also 
be a consequence of a Khinchin-type theorem.

Let $X$, $\Gamma$, $\Cc$, $\Gamma_0$, $T_0$, and $Q$ be as in
Theorem~\ref{thm:main}.  Fix a base point $o \in X$. 
Then for $\eta\in\partial X$, let
$\ell_\eta\colon[0,\infty)\to X$ be the unit speed geodesic ray starting at
$o$ and limiting to $\eta$.

\begin{definition}\label{defn:shadows of subsets} Suppose $K \subset X$.
\begin{itemize}
\item Given $\epsilon > 0$, let $\mathcal N_\epsilon (K)$ denote the   open $\epsilon$-neighborhood of $K$. 
\item Given $T, \epsilon> 0$,  the {\em shadow of depth $T$ and radius $\epsilon$ of $K$} is
  \[
    \Sc_{T,\epsilon}(K)
    := \{\eta\in\partial X \mid \ell_{\eta}([a,b])\subset \mathcal
      N_\epsilon (K) \text{ for some } a,b\in[0,\infty) \text{ with }
    b-a\geq T\}.
  \]
    \end{itemize}
To avoid cumbersome notation, given a function $\phi : [0,\infty) \rightarrow [0,\infty)$,  we let
  \[
    \phi \Sc_{T,\epsilon}(K):=\Sc_{T+\phi(\dist(o,K)),\epsilon}(K).
  \]

  \label{def:shadows}
\end{definition}

Since our base point is fixed, Assumption~\eqref{item:Gamma_counting_assumption} of Theorem~\ref{thm:main} implies there exists a unique Patterson--Sullivan measure $\mu$ for $\Gamma$ of dimension $\delta(\Gamma)$; see 
Theorem~\ref{thm:PS ergodic case}  below.

Fix a function $\phi : [0,\infty) \rightarrow [0,\infty)$ which is slowly
varying (see Section~\ref{s:defns}). Let $[\Cc]$ denote the set of $\Gamma$-translates of $\Cc$ and let
\[
  \Theta^\phi_{T,\epsilon} : =\{\xi \in\partial X\mid \xi\in \phi\Sc_{T,\epsilon}(\alpha\Cc)
    \textrm{ for infinitely many
  }\alpha \Cc\in [\Cc] \}.
\]
The {\em Khinchin series} associated to $\phi$ is 
\[
  K^\phi := \sum_{n\in\mathbb N}
e^{-(\delta(\Gamma)-\delta(\Gamma_0))\phi (n)}
    Q(\phi(n)).
\]

\begin{theorem}[Khinchin-type theorem, see Theorem~\ref{thm:khinchin_1}]
With the notation above,  for any $\epsilon > 0$ and sufficiently large $T > 0$ we have the following dichotomy:
  \begin{enumerate}
    \item If $K^\phi<\infty$, then $\mu(\Theta^\phi_{T,\epsilon} )=0$.
    \item If $K^\phi=\infty$, then $\mu(\Theta^\phi_{T,\epsilon} )=1$.
  \end{enumerate}
  \label{thm:khinchin}
\end{theorem}

We will deduce Theorem~\ref{thm:main} from  Theorem~\ref{thm:khinchin} by 
considering functions of the form $\phi(x)=\kappa \log x$ with $\kappa > 0$, in which case
$$
K^\phi\asymp\sum_{n\in\mathbb N} n^{-\kappa(\delta(\Gamma)-\delta(\Gamma_0))}Q(\log(n)).
$$

An important tool in the proof of Khinchin-type Theorem~\ref{thm:khinchin} 
is a fluctuating density
theorem for 
subset shadows introduced in Definition~\ref{defn:shadows of subsets} which we
call the {\em Subset shadow lemma}:

\begin{theorem}[Subset Shadow Lemma, see Theorem~\ref{thm:shadowing_flats}]
  \label{thm:shadowing_flats_intro} 
With the notation above,   for any $\epsilon > 0$ there exists $C > 1$ such that: if $T \geq 0$ and $\alpha \in \Gamma$, then 
  \[
\frac{1}{C} Q(T) e^{-\delta(\Gamma) (\dist(o,\alpha \Cc)+T)+\delta(\Gamma_0)T}  \leq \mu(\Sc_{T,\epsilon}(\alpha \Cc)) \leq C Q(T) e^{-\delta(\Gamma) (\dist(o,\alpha \Cc)+T)+\delta(\Gamma_0)T}.   \]

\end{theorem}

In particular, Theorem~\ref{thm:shadowing_flats_intro} is used to show
that the Patterson--Sullivan measure of rescaled shadows grows like the
terms of the Khinchin series. The Khinchin-type theorem is then 
proven with a Borel--Cantelli argument. 

\subsection{Historical remarks}

The Khinchin-type theorem and 
logarithm law were first studied in the context of excursions into
noncompact cuspidal regions by 
Sullivan 
for finite volume Kleinian groups \cite{sullivan82}. His approach was based on 
direct connections between the modular group acting on the upper-half plane 
and  the original work of Khinchin for the real line \cite{Khintchine}.

Stratmann--Velani generalized Sullivan's work to the geometrically finite case
\cite{StratmannVelani}, which  
Hersonsky--Paulin extended to 
trees \cite{hersonskypaulin07} and then 
to complete pinched negatively curved Riemannian manifolds with certain growth
assumptions on the parabolic subgroups \cite[Theorems 1.3 and 1.4]{HP2004}.
Related results have been established in the settings of locally symmetric
spaces \cite{kleinbockmargulis98,kleinbockmargulis99} and Gromov hyperbolic
metric spaces \cite{FSU, bray_tiozzo}. See also the survey of Athreya
\cite{athreya_survey}.

Hersonsky--Paulin then proved a Khinchin-type theorem and 
logarithm law for almost-sure spiraling
around convex subsets stabilized by certain convex cocompact subgroups 
in the setting of proper $\CAT(-1)$ spaces
\cite{HP2010}. As discussed above, their results are directly generalized
by our Theorems~\ref{thm:khinchin} and \ref{thm:main}.

  Later work of Ghosh--Nandi~\cite{GN2024} considers the geodesic
  flow on negatively curved Riemannian manifolds and estimates the Hausdorff dimension of the set of vectors $v \in T^1 M$ where 
    \[
    \limsup_{t\to+\infty} \frac{\mathfrak p_{\epsilon}(v,t)}
    {t} \geq \tau
  \]
for some fixed $\tau >0$. It would interesting to know if an analogue of this result extends to the $\CAT(0)$ setting.

\subsection{Structure of the paper}
In Section~\ref{s:preliminaries}, we define the terminology used in Theorem~\ref{thm:main} and recall some useful results about geodesics
and Patterson--Sullivan measures. In Sections~\ref{sec:Morse subsets}
and~\ref{s:convex_cocompact}, we establish some properties of Morse subsets and of convex cocompact groups respectively. 

In Section~\ref{s:shadow_lemma}, we prove the crucial ``Subset Shadow
Lemma'' (Theorem~\ref{thm:shadowing_flats_intro}).
In Section~\ref{s:main}, we prove 
Theorem~\ref{thm:khinchin} using this shadow lemma. Then in
Section~\ref{s:loglaw},
we prove Theorem~\ref{thm:main} using Theorem~\ref{thm:khinchin}.

In the last section of the paper, Section~\ref{sec:Morse Flats}, we
consider periodic Morse flats and show that after thickening, they satisfy
the hypothesis of Theorem~\ref{thm:main}. We conclude with the
proof of Theorem~\ref{thm:nonpositive} in Section~\ref{sec:proof of thm nonpositive}.

\subsection*{Acknowledgements} 
The authors thank Matthew Durham, Autumn Kent, Christopher Leininger,  Caglar Uyanik, and Chenxi Wu for helpful discussions.
Zimmer was partially supported by a Sloan research fellowship and Grant No. DMS-2105580 and DMS-2452068 from the National Science Foundation.

This material is based upon work supported by the National Science Foundation under Grant No. DMS-2424139, while the authors were in residence at the Simons Laufer Mathematical Sciences Institute in Berkeley, California, during the Spring 2026 semester.

\section{Preliminaries}
\label{s:preliminaries}

\subsection{Notation and conventions}
\label{s:notation}
Given positive-valued functions $f,g$, we write $f\lesssim g$ if there exists a constant $C>0$ so that $f\leq C g$, and $f\asymp g$ if $f\lesssim g$ and $g\lesssim f$.

Throughout the paper, $X$ will be a proper $\CAT(0)$ space, $\partial X$ will denote the geodesic compactification of $X$, and    $\Bc_r(x)$ will denote the ball of radius $r$ about a point $x \in X$. 

 A geodesic is a map $\ell : I \rightarrow X$ of an interval $I \subset \Rb$ into $X$ such that 
$$
\dist(\ell(s), \ell(t)) = \abs{t-s}
$$
for all $s,t \in I$ (in particular, we always assume that our geodesics are unit speed). We will frequently identify a geodesic with its image.

  The natural notion of convergence for geodesics is locally uniform convergence. A sequence of intervals $(I_n)$ \emph{converges} to an interval $I$ when the following two properties hold:
\begin{enumerate}
\item If $J \subset I$ is a bounded closed interval, then $J \subset I_n$ for $n$ sufficiently large.
\item If $J$ is a bounded closed interval and $J \subset I_n$ for infinitely many $n$, then $J \subset I$. 
\end{enumerate}
Then a sequence of maps $(f_n : I_n \rightarrow X)$ \emph{converges locally uniformly} to a map $f : I \rightarrow X$ if for every bounded closed interval $J \subset I$ the restrictions $f_n|_J$ converges uniformly to $f|_J$.

Two geodesic lines $\ell_1, \ell_2 : \Rb \rightarrow X$ are \emph{parallel} if
$$
t \in \Rb \mapsto \dist(\ell_1(t), \ell_2(t))\in [0,\infty)
$$
is bounded, which implies that this function is constant (see \cite[Chapter II.2 Theorem 2.13]{BH}).  

Since $X$ is $\CAT(0)$, every two points $x,y \in X$ are joined by a unique geodesic (see~\cite[Part II, Proposition 1.4]{BH})
and we denote its image by $[x,y]$. Likewise, for $x \in X$ and $\xi \in \partial X$, we let $[x,\xi)$ denote the image of the unique geodesic ray starting at $x$ and limiting to $\xi$ (for uniqueness see~\cite[Part II, Proposition 8.2]{BH}).

Given a subgroup $G < \Isom(X)$, we use $\Lambda(G) \subset \partial X$ to denote the limit set of $G$, that is the set of an accumulation points of the orbit $G \cdot o$ in $\partial X$ for some (any) $o \in X$.

\subsection{Key definitions}
\label{s:defns}
We continue to assume that $X$ is a proper $\CAT(0)$ space and now introduce the main definitions from the assumptions of
Theorem~\ref{thm:main}.
\begin{itemize}
\item A geodesic $\ell : \Rb \rightarrow X$ has \emph{rank one} if $\ell$
  does not bound a half plane, that is $\ell$ does not extend to an
  isometric embedding of $\Rb \times [0,\infty)$. A \emph{rank one
  isometry} of $X$ is an isometry that translates along a rank one geodesic. A subgroup $\Gamma \subset \Isom(X)$ has \emph{rank one} if it contains a rank one isometry and is \emph{non-elementary} if its limit set $\Lambda(\Gamma)$ has at least three points. 
\item A convex subset $\Cc \subset X$ is \emph{Morse} if for every $A
\geq 1$, $B \geq 0$ there exists $D = D(A,B) > 0$ such that: whenever $\ell
\colon[a,b] \rightarrow X$ is a $(A,B)$-quasi-geodesic with endpoints in $\Cc$ we have 
$$
\ell \subset \Nc_D(\Cc).
$$
\item A subgroup $H < G$ is \emph{almost malnormal} if $gHg^{-1} \cap H$ is finite for all $g \in G \smallsetminus H$. 
\item A  function $\phi\colon [0,\infty)\to[0,\infty)$ is 
{\em slowly varying}
if there exist constants $B,A>0$ such that 
$\abs{x-y} \leq B$ implies
$\abs{\phi(x)-\phi(y)}\leq A$. 
\end{itemize}

\begin{remark} There are multiple definitions of slowly varying functions 
in the literature and we note that $\phi$ is slowly varying as defined here if and only if $e^\phi$ is slowly
varying as defined in  \cite{HP2010}.
\end{remark}

\subsection{Convexity} We continue to assume that $X$ is a proper $\CAT(0)$ space. In this section we recall some important convexity properties of $\CAT(0)$ spaces. 

Recall that the \emph{Hausdorff pseudo-distance} between two subsets $A,B \subset X$ is 
$$
\dist^{\rm Haus}(A,B) = \max\left\{ \sup_{a\in A} \dist(a,B), \, \sup_{b \in B} \dist(b, A) \right\}. 
$$

\begin{lemma}\label{lem:HD_between_geod} If $x_1,x_2, y_1, y_2 \in X$, then 
$$
\dist^{\rm Haus}([x_1,x_2], [y_1,y_2]) \leq \max( \dist(x_1,y_1), \dist(x_2,y_2)). 
$$
In particular, if $x, y \in X$ and $\eta \in \partial X$, then 
$$
\dist^{\rm Haus}([x,\eta), [y,\eta)) \leq \dist(x,y).
$$
\end{lemma}

\begin{proof} The first assertion is  \cite[Chapter II.2 Proposition 2.2]{BH}. For the ``in particular'' part, fix $z_n \rightarrow \eta$. Since unit speed parametrizations of $[x,z_n)$ and $[y,z_n)$ converge locally uniformly to unit speed parametrizations of $[x,\eta)$ and $[y,\eta)$ respectively, we have 
$$
\dist^{\rm Haus}([x,\eta), [y,\eta)) \leq \limsup_{n \rightarrow \infty} \dist^{\rm Haus}([x,z_n], [y,z_n]) \leq \dist(x,y)
$$
by the first assertion. 
\end{proof}

\begin{lemma} \cite[Chapter II.2 Corollary  2.5]{BH}  \label{lem:following_flats}
  For any convex set $\Cc$ and any geodesic $\ell \colon [a,b] \rightarrow X$ the function 
  $$
  t \mapsto \dist(\ell(t), \Cc)
  $$
  is convex. Hence,   if $\dist(\ell(a),\Cc),\dist(\ell(b),\Cc) <  \epsilon$, then 
  $$
  \ell([a,b])\subset \Nc_\epsilon(\Cc).
  $$
\end{lemma}

\subsection{Patterson--Sullivan measures}
\label{s:patterson_sullivan}
We continue to assume that $X$ is a proper $\CAT(0)$ space. Fix a base point $o\in X$ and for $\xi \in \partial X$, let $\ell_\xi \colon [0,\infty) \rightarrow X$ denote the geodesic ray starting at $o$ and limiting to $\xi$. Then let 
$$
b_\xi(x) = \lim_{t \rightarrow \infty} \big(\dist(\ell_{\xi}(t), x) - t\big)
$$
denote the Busemann function based at $\xi$. 

Given a discrete subgroup $\Gamma < \Isom(X)$, a \emph{Patterson--Sullivan measure for $\Gamma$ of dimension $\delta$} is a  Borel probability measures $\mu$ on $\partial X$ such that for every $\gamma \in \Gamma$ the measures $\mu$, $\gamma_*\mu$ are absolutely continuous and 
$$
\frac{d\gamma_*\mu}{d\mu}(\xi)= e^{-\delta b_{\xi}(\gamma^{-1}o)} \quad \text{for $\mu$-a.e. $\xi$}. 
$$

Using Patterson's original construction for Fuchsian groups~\cite{Patterson}, one has the following existence result.

\begin{proposition}[Patterson]\label{prop:PS exist}
If $\Gamma < \Isom(X)$ is discrete and $\delta(\Gamma) < \infty$, then there exists a Patterson--Sullivan measure for $\Gamma$ of dimension $\delta(\Gamma)$. 
\end{proposition}

As a consequence of Link's~\cite{Link2018} version of the Hopf--Tsuji--Sullivan dichotomy, when the Poincar\'e series diverges at the critical exponent this measure is unique. 

\begin{theorem}[Link]\label{thm:PS ergodic case} Suppose $\Gamma < \Isom(X)$ is a non-elementary rank one discrete subgroup, $\delta(\Gamma) < +\infty$, and 
  
\begin{equation}
  \label{eqn:div_type}
  \sum_{\gamma \in \Gamma} e^{-\delta(\Gamma) \dist(o,\gamma o)} = \infty. 
\end{equation}
Then there exists a unique Patterson--Sullivan measure $\mu$ for $\Gamma$ of dimension $\delta(\Gamma)$ and the $\Gamma$-action on $(\partial X, \mu)$ is ergodic (i.e. every $\Gamma$-invariant measurable set has either zero or full measure). 
\end{theorem} 

Note that Assumption~\eqref{item:Gamma_counting_assumption} in Theorem~\ref{thm:main} 
implies \eqref{eqn:div_type} in Theorem~\ref{thm:PS ergodic case}.

\begin{proof}[Proof of Theorem~\ref{thm:PS ergodic case}] Ergodicity of any  Patterson--Sullivan measure for $\Gamma$ of dimension $\delta(\Gamma)$  follows from Theorem 10.1 in~\cite{Link2018} and Proposition 4 in~\cite{LinkPicaud2016}. The later reference assumes that $X$ is a manifold, but the same argument works in general.

  Uniqueness follows from ergodicity exactly as in the classical
real hyperbolic case (see ~\cite[pg.\ 181]{sullivan79}). Suppose $\mu, \mu'$
are Patterson--Sullivan measures for $\Gamma$ of dimension
$\delta(\Gamma)$. Then $\nu : = \frac{1}{2}(\mu+\mu')$ is also a
Patterson--Sullivan measure for $\Gamma$ of dimension $\delta(\Gamma)$.
Furthermore, the measurable functions $\frac{d\mu}{d\nu}, \frac{d\mu'}{d\nu}$ are both $\Gamma$-invariant. Since $\Gamma$ acts ergodically on $(\partial X, \nu)$, these functions are constant. Finally, since $\mu$, $\mu'$, $\nu$ are probability measures, $\mu=\nu=\mu'$. 
\end{proof} 

Given $r > 0$ and $x,y \in X$ the associated \emph{shadow} is 
$$
\Oc_r(x,y) : = \{ \xi \in \partial X : [x,\xi) \cap \Bc_r(y) \neq \emptyset \}.
$$

\begin{proposition}[The Shadow Lemma]\label{prop:shadow lemma} Suppose $\Gamma < \Isom(X)$ is a discrete subgroup and $\mu$ is a Patterson--Sullivan measure for $\Gamma$ of dimension $\delta$. For any $r > 0$, there exists $C_1=C_1(r) > 0$ such that:
$$
 \mu(\Oc_r(o,\gamma o)) \leq C_1 e^{-\delta \dist(o,\gamma o)} \quad \text{for all $\gamma \in \Gamma$}. 
$$
If, in addition, $\Gamma$ is a non-elementary rank one discrete subgroup, then for any $r > 0$ sufficiently large, there exists $C_2=C_2(r) > 0$ such that: 
$$
C_2 e^{-\delta \dist(o,\gamma o)} \leq \mu(\Oc_r(o,\gamma o)) \quad \text{for all $\gamma \in \Gamma$}. 
$$
\end{proposition} 

\begin{proof} For the first assertion see Lemme 1.3 in~\cite{roblin}, which
  assumes that $X$ is $\CAT(-1)$, but the same argument works for $\CAT(0)$
  spaces. 

    For the second assertion, by Proposition 3
    in~\cite{LinkPicaud2016} (which assumes that $X$ is a manifold, but the
    same argument works in general), there exists $r_0 > 0$ such that for any $r \geq r_0$ there exists $C_2'(r) > 0$ such that: 
$$
C_2'(r) e^{-\delta \dist(o,\gamma o)} \leq \mu(\Oc_r(o,\gamma o)) 
$$
for all $\gamma \in \Gamma$ with $\dist(o,\gamma o) > 2r$. 

Now fix $r \geq 2r_0+1$. If $\dist(o,\gamma o) > r-1$, then
$$
\mu(\Oc_r(o,\gamma o)) \geq \mu(\Oc_{(r-1)/2}(o,\gamma o)) \geq C_2'\left( \frac{r-1}{2}\right) e^{-\delta \dist(o,\gamma o)}.
$$
Otherwise, if $\dist(o,\gamma o) \leq r-1$, then $\Oc_r(o,\gamma o) = \partial X$ and so 
$$
\mu(\Oc_r(o,\gamma o)) = 1 \geq e^{-\delta \dist(o,\gamma o)}.
$$
Thus, the second assertion holds with 
\begin{equation*}
C_2(r) : = \min\left\{ C_2'\left( \frac{r-1}{2}\right) , 1 \right\}. \qedhere
\end{equation*}

\end{proof}

\section{Morse subsets}\label{sec:Morse subsets}

For the rest of the section suppose that $X$ is a proper $\CAT(0)$ space. In this section we establish some properties of Morse subsets.

We start by stating some equivalent characterizations in terms of the closest point projection map. Given a convex subset $\Cc$, let $\pi_{\Cc} \colon X \rightarrow \Cc$ denote the closest point projection (which is well defined, see for instance~\cite[Chapter II.2 Proposition  2.4]{BH}). A convex subset $\Cc \subset X$ is \emph{strongly contracting} if there exists $D > 0$ such that: if $\dist(x,\Cc) = r$, then 
$$
{\rm diam} \, \pi_{\Cc}(\Bc_r(x)) \leq D.
$$

Also recall that a geodesic triangle 
 $$
 [x,y] \cup [y, z] \cup [z,x]
 $$
 is \emph{$\sigma$-slim} if any side is contained the the $\sigma$-neighborhood of the other two sides.

Building upon work of Cashen~\cite{Cashen2020} and Charney--Sultan~\cite{CS2015}, we have the following characterization of Morse subsets in $\CAT(0)$ spaces.

\begin{theorem}\label{thm:equiv of Morse} If $\Cc \subset X$ is convex, then the following are equivalent: 
\begin{enumerate} 
\item $\Cc$ is a Morse subset. 
\item $\Cc$ is strongly contracting.
\item There exists $\sigma \geq 0$ 
  such that: if $x \in X$ and $y \in \Cc$, then 
$$
\dist(\pi_{\Cc}(x), [x,y]) \leq \sigma. 
$$
\item There exists $\sigma' \geq 0$ 
  such that:  if $x \in X$, $y_1,y_2 \in \Cc$, and $\pi_{\Cc}(x) \in [y_1,y_2]$, then the geodesic triangle 
 $$
 [x,y_1] \cup [y_1, y_2] \cup [y_2,x]
 $$
 is $\sigma'$-slim.
\end{enumerate}
\end{theorem} 

\begin{remark}
Note that since closest point projections are equivariant with respect to isometries, if $\sigma$ satisfies (3) for $\Cc$ then it does as well for any translate of $\Cc$ by an isometry of $X$. 
\end{remark}

\begin{figure}[h!]
  \centering
%
%
%

\begin{tikzpicture}[scale=1]

  \draw (0,0)--(4,2)--(11,2)--(7+.2,-1) -- node[pos=.05,yshift=20pt]
  {$\Cc$} cycle;

  \draw (4.5,5) coordinate (x) -- 
  coordinate[pos=.6] (xp) 
  (4.5,1) coordinate (pix) -- 
  coordinate[pos=.6] (yp) (7,.8) coordinate
  (y) to [out=150, in=-80] coordinate[pos=.2] (ypp) coordinate[pos=.4]
  (mid)
  coordinate[pos=.53] (xpp) (x);

  \begin{scope}[dashed]
    \draw[] (mid)--(pix) node[midway,xshift=5pt, yshift=-2pt] {$\sigma$};
    \draw[] (yp)--coordinate[midway] (midy) (ypp) ; 
    \draw[] (xp)--coordinate[midway] (midx) (xpp) ; 
  \end{scope}

  \draw (xp)+(-.2,-.1) coordinate (nxp);
  \draw (pix)+(-.2,.1) coordinate (nxpix);
  \draw (yp)+(-.1,-.2) coordinate (nyp);
  \draw (pix)+(.1,-.2) coordinate (nypix);

  \draw [decorate,decoration={brace,amplitude=5pt}]
  (nxpix)--(nxp)
  node [midway,xshift=-10pt,yshift=0pt] {$\sigma$};
  \draw [decorate,decoration={brace,amplitude=5pt}]
  (nyp)--(nypix)
  node [midway,yshift=-10pt,xshift=0pt] {$\sigma$};

  \draw(midy)+(.1,0) coordinate (nmidy);
  \draw(midx)+(0,.1) coordinate (nmidx);

  \draw[->] (y)+(0,.7) node[above right] {$\leq \sigma$}
  to [out=-120, in=-20]
  (nmidy);

  \draw[->] (xpp)+(.1,.9) node[right] {$\leq\sigma$}
  to [out=200, in=110] (nmidx);

  \def\s{.05}
  \draw[fill] (x) circle (\s) node[above left] {$x$};
  \draw[fill] (pix) circle (\s) node[below left] {$\pi_{\Cc}(x)$};
   \draw[fill] (y) circle (\s) node[below right] {$y$};
   \draw[fill] (xp) circle (\s) node[above left] {$x'$};
   \draw[fill] (yp) circle (\s) node[below right] {$y'$};
   \draw[fill] (xpp) circle (\s) node[above right] {$x''$};
   \draw[fill] (ypp) circle (\s) node[above right] {$y''$};
\end{tikzpicture}

  \caption{The configuration of points in the proof of the implication (2) $\Rightarrow$ (3) in
  Theorem~\ref{thm:equiv of Morse}.}
  \label{fig:equiv_of_Morse}
\end{figure}

\begin{proof}[Proof of Theorem~\ref{thm:equiv of Morse}]
The equivalence (1) $\Leftrightarrow$ (2) is a theorem of
Cashen~\cite{Cashen2020}. Charney--Sultan~\cite[Theorem 2.14]{CS2015}
proved the above equivalences for geodesics (along with several other
equivalent conditions) and their argument for the implication (3)
$\Rightarrow$ (2) taken verbatim works for convex
subsets. Thus, it suffices to prove (2) $\Rightarrow$ (3) and (3) $\Leftrightarrow$ (4).

  (4) $\Rightarrow$ (3): We claim that $\sigma:=2\sigma'$ suffices. Fix $x \in X$ and $y \in \Cc$. By (4) the geodesic triangle 
 $$
 T:=[x,\pi_{\Cc}(x)] \cup [\pi_{\Cc}(x), y] \cup [y,x]
 $$
 is $\sigma'$-slim. If $\dist(\pi_{\Cc}(x),x) < \sigma'$, then there is
 nothing to prove. Thus, we can assume that $\dist(\pi_{\Cc}(x),x) \geq \sigma'$. Fix $u \in[\pi_{\Cc}(x), x]$ with $\dist(\pi_{\Cc}(x),u) =\sigma'$. Then $\pi_{\Cc}(u) = \pi_{\Cc}(x)$ and so 
 $$
 \dist(u,  [\pi_{\Cc}(x), y]) \geq \dist(u, \Cc) = \sigma'.
 $$
Then, since $T$ is $\sigma'$-slim, we must have $\dist(u, [x,y]) <
\sigma'$. Thus, 
$$
\dist(\pi_{\Cc}(x), [x,y]) \leq \dist(\pi_{\Cc}(x),u)+\dist(u, [x,y]) <  2\sigma'. 
$$

(3) $\Rightarrow$ (4): We claim that any $\sigma'>2\sigma$ suffices. Fix $x \in X$ and $y_1, y_2 \in \Cc$ with  $\pi_{\Cc}(x) \in [y_1,y_2]$. By (3) there exists $u_i \in [x,y_i]$ with $\dist(u_i,\pi_{\Cc}(x)) \leq \sigma$. By Lemma~\ref{lem:HD_between_geod}, 
\begin{align*}
\dist^{\rm Haus}& ( [y_1, u_1], [y_1, \pi_{\Cc}(x)]) \leq \sigma, \quad \dist^{\rm Haus}( [y_2, u_2], [y_2,\pi_{\Cc}(x)]) \leq \sigma, \\
& \text{and} \quad \dist^{\rm Haus}( [x, u_1], [x,u_2]) \leq 2\sigma.
\end{align*}
Hence, the geodesic triangle is $\sigma'$-slim for any $\sigma' > 2\sigma$.

(2) $\Rightarrow$ (3): Suppose $\Cc$ is strongly contracting with constant $D$. We claim that (3) is true for $\sigma = 19D$. Suppose not. Then there exist $\sigma > 19D$, $x \in X$, and $y \in \Cc$ with 
$$
\dist(\pi_{\Cc}(x), [x,y]) \geq \sigma. 
$$
By replacing $x$ with a point on the geodesic $[\pi_{\Cc}(x), x]$ we can assume that 
$$
\dist(\pi_{\Cc}(x), [x,y]) = \sigma. 
$$
Let $x' \in [\pi_{\Cc}(x),x]$ be the point with $\dist(\pi_{\Cc}(x),x') =
\sigma$. By Lemma~\ref{lem:HD_between_geod}, there exists $x'' \in [x,y]$ with 
$$
\dist(x',x'') \leq \sigma. 
$$
Likewise, let $y' \in [\pi_{\Cc}(x),y]$ be the point with $\dist(\pi_{\Cc}(x),y') = \sigma$ and fix $y'' \in [x,y]$ with 
$$
\dist(y',y'') \leq \sigma. 
$$
See Figure~\ref{fig:equiv_of_Morse}.

Since $x' \in [\pi_{\Cc}(x),x]$, we have $\pi_{\Cc}(x') = \pi_{\Cc}(x)$.
Then by the contracting property, 
$$
  \dist (\pi_{\Cc}(x),\pi_{\Cc}(x''))=\dist(\pi_{\Cc}(x'), \pi_{\Cc}(x'')) \leq D.
$$
Since 
$$
\dist(x'',\pi_{\Cc}(x))\geq \dist(\pi_{\Cc}(x), [x,y]) = \sigma, 
$$
we have 
$$
\dist(x'',\Cc) = \dist(x'', \pi_{\Cc}(x'')) \geq\dist(x'',\pi_{\Cc}(x))-  \dist (\pi_{\Cc}(x),\pi_{\Cc}(x'')) \geq \sigma - D. 
$$
Next pick $x_1,\dots, x_6$ in order along $[x'',y]$ such that $x_1 = x''$ and
$$
\dist(x_{j}, x_{j+1}) = \sigma - jD
$$
for $j \geq 1$. By the contracting property and induction, 
$$
\dist(\pi_{\Cc}(x), \pi_{\Cc}(x_j)) \leq j D \quad \text{and} \quad \dist(x_j,\Cc) \geq \sigma - jD > D
$$
for $j=1,\dots, 6$. (Notice that since $\sigma > 19D$, the estimate $\dist(x_j, y) \geq \dist(x_j,\Cc) \geq \sigma - jD$ implies that $x_1,\dots, x_6$ do indeed exist). 

Now
$$
\dist(x_1,y'') \leq \dist(x_1,x') + \dist(x',\pi_{\Cc}(x)) + \dist(\pi_{\Cc}(x), y') + \dist(y',y'') \leq 4 \sigma
$$
and 
$$
\dist(x_1,x_6) = \sum_{j=1}^5 \sigma - j D = 5 \sigma - 15D > 4\sigma. 
$$
Hence, $x_6 \in [y'', y]$, and by Lemma~\ref{lem:HD_between_geod},
$$
\dist(x_6, [y',y]) \leq \dist^{\rm Haus}([y'',y], [y',y]) \leq \sigma. 
$$
Since 
$$
 \dist(x_6,\Cc) \geq \sigma - 6D,
 $$
there exists $z \in [x_6, \pi_{[y',y]}(x_6)]$ with $\dist(z, x_6) = \sigma-6D$. Then 
$$
\dist(z,\pi_{[y',y]}(x_6)) \leq 6D.
$$
Furthermore, since $\dist(x_6,\Cc)\geq \sigma -6D$, the contracting property implies that
$$
\dist(\pi_{\Cc}(z), \pi_{\Cc}(x_6)) \leq D
$$
and so 
$$
\dist(\pi_{\Cc}(x),\pi_{\Cc}(z)) \leq \dist(\pi_{\Cc}(x), \pi_{\Cc}(x_6))+\dist(\pi_{\Cc}(x_6), \pi_{\Cc}(z)) \leq  7D. 
$$
On the other hand, 
$$
\dist( \pi_{[y',y]}(x_6), \pi_{\Cc}(z)) \leq \dist( \pi_{[y',y]}(x_6),z) + \dist(z,\pi_{\Cc}(z)) \leq 2 \dist( \pi_{[y',y]}(x_6),z) \leq 12D. 
$$
Thus, 
$$
19D < \sigma = \dist(y', \pi_{\Cc}(x)) \leq \dist(\pi_{[y',y]}(x_6),\pi_{\Cc}(x)) \leq 19D
$$
and we have a contradiction. 
\end{proof} 

As a consequence of (3) in Theorem~\ref{thm:equiv of Morse}, Morse subsets have the following properties. 

\begin{proposition}\label{prop:projection to Morse subset} 
  Let $\Cc \subset X$ be a convex set.  If $\Cc$ is a Morse subset and
  $\sigma \geq 0$ is as in Theorem~\ref{thm:equiv of Morse}, then:
\begin{enumerate}
\item  If $x \in X$ and $y \in \Nc_\epsilon(\Cc)$, then 
$$
\dist(\pi_{\Cc}(x), [x,y]) < \epsilon + \sigma
$$
and
 $$
 \dist(x,y) > \dist(x, \pi_{\Cc}(x)) + \dist( \pi_{\Cc}(x), y) -2\epsilon-2\sigma. 
 $$
 \item If $x \in X$, and $y \in \Nc_\epsilon(\Cc)$, then the first point   of
   intersection between $[x,y]$ and $\overline{\Nc_{\epsilon+\sigma}(\Cc)}$   is within a distance $3\epsilon+3\sigma$ of $\pi_{\Cc}(x)$. 
\end{enumerate}
\end{proposition}

\begin{proof} Fix $x \in X$ and $y \in \Nc_\epsilon(\Cc)$.

Fix $y' \in \Cc$ with $\dist(y,y') < \epsilon$. Then fix $u' \in [x,y']$
with $\dist(\pi_{\Cc}(x),u') \leq \sigma$. By
Lemma~\ref{lem:HD_between_geod}, there exists $u \in [x,y]$ with $\dist(u,u') < \epsilon$. Then  
$$
\dist(\pi_{\Cc}(x), [x,y]) \leq \dist(\pi_{\Cc}(x),u')+ \dist(u',u) < \epsilon + \sigma
$$
and by the triangle inequality
 $$
 \dist(x,y)=\dist(x,u)+\dist(u,y) \geq \dist(x, \pi_{\Cc}(x)) + \dist( \pi_{\Cc}(x), y) -2\epsilon-2\sigma. 
 $$
We conclude (1) is true. 

Let $v$ be the first point of intersection between $[x,y]$ and 
$\overline{\Nc_{\epsilon+\sigma}(\Cc)}$. 
Suppose for a contradiction that $\dist(v,\pi_{\Cc}(x)) > 3\epsilon+3\sigma$.
Note that
$u \in \Nc_{\epsilon+\sigma}(\Cc)$ 
and so $v \in [x,u]$. Since 
$$
\dist(u, \pi_{\Cc}(x)) < \epsilon + \sigma, 
$$
by the triangle inequality,
$$
\dist(u,v) > 2\epsilon+2\sigma.
$$
Therefore, 
\begin{align*}
\dist(x,\Cc) & = \dist(x,\pi_{\Cc}(x)) \geq \dist(x,u) - (\epsilon+\sigma)= \dist(x,v)+\dist(v,u) -\epsilon-\sigma \\
& > \dist(x,v)+(2\epsilon+2\sigma)-\epsilon-\sigma > \dist(x,v) + \epsilon+\sigma.
\end{align*}
However, since $v \in \overline{\Nc_{\epsilon+\sigma}(\Cc)}$, 
$$
\dist(x,\Cc) \leq \dist(x,v) +\epsilon+\sigma
$$
and so we have a contradiction. Thus, (2) is true. 

\end{proof}

The next lemma is used in the proof of Proposition~\ref{prop:Cond 2 holds after thickening} below. 

\begin{lemma}\label{lem:bounded implies really bounded} Suppose $\Cc \subset X$ is Morse and $\sigma > 0$ satisfies Theorem~\ref{thm:equiv of Morse}. If $\ell$ is a geodesic line in $X$ with $\ell \subset \Nc_R(\Cc)$ for some $R >0$, then $\ell \subset \overline{\Nc_\sigma(\Cc)}$ and $\ell$ is parallel to a geodesic line in $\Cc$. 
\end{lemma} 

\begin{proof} For the first assertion,   fix $z \in \ell$. Then fix a unit speed parametrization of $\ell$ with $\ell(0)=z$.

  Let $x : = \ell(-R-\sigma)$ and $y_n : = \pi_{\Cc}(\ell(n))$. Then let $\ell_n : [-R-\sigma,b_n] \rightarrow X$ be the geodesic joining $x$ to $y_n$. Since $\dist(y_n, \ell(n)) \leq R$, we have 
$$
\xi: = \lim_{n \rightarrow \infty} y_n = \lim_{t \rightarrow\infty} \ell(t)
$$
in $\partial X$. Since $\ell|_{[-R-\sigma, \infty)}$ is the unique geodesic ray starting at $x$ and limiting to $\xi$, the geodesics $\ell_n$ converge locally uniformly to $\ell|_{[-R-\sigma,\infty)}$.

  Recalling that $\ell_n$ parametrizes $[x,y_n]$ and $y_n \in \Cc$, by Proposition~\ref{prop:projection to Morse subset} part (1), there exists $t_n \in [-R-\sigma, b_n]$ such that 
$$
\dist( \pi_{\Cc}(x),\ell_n(t_n)) =\dist( \pi_{\Cc}(x),[x,y_n]) \leq \sigma. 
$$
Since $\dist(x, \pi_{\Cc}(x)) \leq R$ and $\ell_n(-R-\sigma) =x$, we have $t_n \leq 0$. Then Lemma~\ref{lem:following_flats} implies that 
$$
\ell_n([0,b_n]) \subset \ell_n([t_n,b_n]) \subset \overline{\Nc_\sigma(\Cc)}
$$
and so 
\begin{equation*}
z=\ell(0) = \lim_{n \rightarrow \infty} \ell_n(0) \in  \overline{\Nc_\sigma(\Cc)}.
\end{equation*}
Since $z \in \ell$ was arbitrary, the first assertion is true.

For the second assertion, for each $n \in \Nb$ let $z_n^\pm := \ell(\pm n)$ and let $\tilde \ell_n : [c_n, d_n] \rightarrow \Cc$ be the geodesic joining $\pi_{\Cc}(z_n^-)$ and $\pi_{\Cc}(z_n^+)$. Since $\dist(z_n^\pm, \pi_{\Cc}(z_n^\pm)) \leq \sigma$, Lemma~\ref{lem:HD_between_geod} implies that 
$$
\tilde \ell_n \subset  \overline{\Nc_\sigma(\ell([-n,n]))}.
$$
Thus, we can parametrize $\tilde \ell_n$ so that $\dist(\tilde \ell_n(0), \ell(0)) \leq \sigma$. Then the sequence $( \tilde \ell_n(0) )$ is bounded in $X$. Hence, there exists a subsequence $(\tilde \ell_{n_j})$ which converges locally uniformly to a geodesic line $\tilde \ell : \Rb \rightarrow \Cc$ with 
$$
\tilde \ell \subset  \overline{\Nc_\sigma(\ell)}.
$$
Thus, $\ell$ and $\tilde \ell$ are parallel. 
\end{proof} 

Recall that the convex hull of a subset $A \subset X$ is the smallest convex subset of $X$ containing $A$. 

\begin{proposition}\label{prop:Cond 2 holds after thickening} Suppose $\Cc
  \subset X$ is Morse and $\sigma > 0$ satisfies Theorem~\ref{thm:equiv of
  Morse}. If $\Cc'$ is the convex hull of   the union of
$\Cc$ and all geodesic lines in $X$ parallel to a geodesic line in $\Cc$, then 
\begin{enumerate}
\item $\Cc' \subset \overline{\Nc_{\sigma}(\Cc)}$. 
\item $\Cc'$ contains all geodesic lines in $X$ parallel to a geodesic line in $\Cc'$. 
\end{enumerate} 
\end{proposition} 

\begin{proof} Lemma~\ref{lem:bounded implies really bounded} implies that $\overline{\Nc_{\sigma}(\Cc)}$ contains $\Cc$ and all geodesic lines in $X$ parallel to a geodesic line in $\Cc$. Since $ \overline{\Nc_{\sigma}(\Cc)}$ is convex,  see Lemma~\ref{lem:following_flats}, we then have $\Cc' \subset \overline{\Nc_{\sigma}(\Cc)}$. 

For (2), suppose that $\ell$ is a geodesic line in $X$ parallel to a geodesic line in $\Cc'$. Then there exists some $R > 0$ such that 
$$
\ell \subset \Nc_R(\Cc'),
$$
which implies that 
$$
\ell \subset \Nc_{R+\sigma+1}(\Cc).
$$
Thus, Lemma~\ref{lem:bounded implies really bounded} implies that $\ell$ is parallel to a geodesic line in $\Cc$ and hence $\ell \subset \Cc'$. 
\end{proof}

\section{Convex cocompact subgroups} 
\label{s:convex_cocompact}

 In this section we establish some properties of convex cocompact subgroups. To that end, for the rest of the section assume that:
\begin{itemize}
\item  $X$ is a proper $\CAT(0)$ metric space, 
\item  $\Gamma < \Isom(X)$ is a non-elementary rank one discrete group with $\delta(\Gamma) < \infty$, 
\item $\Gamma_0 < \Gamma$ is a subgroup which acts cocompactly on a
  closed convex subset $\Cc \subset X$, and 
\item   $o\in X$.
\end{itemize}

Using an argument from~\cite{Coornaert}, we first observe that the Poincar\'e series for $\Gamma_0$ 
diverges at its critical exponent. 

\begin{proposition}\label{prop:divergence in cocompact case} 
$\sum_{\gamma \in \Gamma_0} e^{-\delta(\Gamma_0) \dist(o,\gamma o)} = \infty. $
\end{proposition} 

\begin{proof} By Proposition~\ref{prop:PS exist}, there exists a Patterson--Sullivan measure $\mu_0$ for $\Gamma_0$ of dimension $\delta(\Gamma_0)$ supported on $\partial \Cc \subset \partial X$.   Since $\Gamma_0$ acts cocompactly on $\Cc$,  there exists 
 $r \in \Nb$ such that $\Gamma_0 \cdot \Bc_r(o) =\Cc$. By the Shadow Lemma
 (Proposition~\ref{prop:shadow lemma}), there exists $C > 0$ such that 
$$
\mu_0(\Oc_r(o,\gamma o)) \leq C e^{-\delta(\Gamma_0)\dist(o,\gamma o)}
$$
for all $\gamma \in \Gamma_0$. Using this estimate and the fact that $\Gamma_0$ acts cocompactly on $\Cc$, we can argue exactly as in Th\'eor\`eme 7.2 and Corollaire 7.3 in~\cite{Coornaert} to deduce that 
\begin{equation*}
\sum_{\gamma \in \Gamma_0} e^{-\delta(\Gamma_0) \dist(o,\gamma o)} = \infty. \qedhere
\end{equation*}
\end{proof} 

Using a result from~\cite{Link2018}, we establish a critical exponent gap whenever $\Gamma_0$ has infinite index in $\Gamma$. 

\begin{proposition} 
  \label{prop:entropy_gap}
  $\Gamma_0 < \Gamma$ has infinite index if and only if $\delta(\Gamma_0) < \delta(\Gamma)$. 
\end{proposition} 

\begin{proof} If $\Gamma_0 < \Gamma$ has finite index, then $\delta(\Gamma_0) = \delta(\Gamma)$. Hence, $\delta(\Gamma_0) < \delta(\Gamma)$ implies that   $\Gamma_0 < \Gamma$ has infinite index. 

For the other direction, assume that $\Gamma_0 < \Gamma$ has infinite index. Recall that $\Lambda(G)$ denotes the limit set of a subgroup $G < \Isom(X)$. Since $\Gamma_0$ acts cocompactly on $\Cc$, we have $\Lambda(\Gamma_0) = \partial \Cc$. Then, using Proposition~\ref{prop:divergence in cocompact case} and a result of Link~\cite[Proposition 8]{Link2018}, it suffices to show that $\partial \Cc \subsetneq \Lambda(\Gamma)$. Since $\Gamma_0 < \Gamma$ has infinite index and $\Gamma_0$ acts cocompactly on $\Cc$, for each $n \geq 1$ there exists $\gamma_n \in \Gamma$ with $\dist(\gamma_n o, \Cc) \geq n$. Translating by elements of $\Gamma_0$ and passing to a subsequence we can suppose that 
$$
\pi_{\Cc}(\gamma_n o) \rightarrow x \in \Cc \quad \text{and} \quad \gamma_n o \rightarrow \xi \in \Lambda(\Gamma). 
$$

Suppose for a contradiction that $\xi \in \partial \Cc$. Then, by
convexity, $[x, \xi) \subset \Cc$. Fix $y \in (x, \xi)$ and $y_n \in
(\pi_{\Cc}(\gamma_n o), \gamma_n o)$ such that $y_n \rightarrow y$. Since
$y \in \Cc$ and $y_n \rightarrow y$ we have $\pi_{\Cc}(y_n) \rightarrow y$.
On the other hand, $y_n \in  (\pi_{\Cc}(\gamma_n o), \gamma_n o)$ and so we
have $\pi_{\Cc}(y_n) = \pi_{\Cc}(\gamma_n o)$. Thus, 
$$
y = \lim_{n \rightarrow \infty} \pi_{\Cc}(y_n) =\lim_{n \rightarrow \infty}  \pi_{\Cc}(\gamma_n o)=x,
$$
which is a contradiction. We conclude $\xi \notin  \partial \Cc$ and so $\partial \Cc \subsetneq \Lambda(\Gamma)$.
\end{proof}

\begin{proposition} If $\Gamma_0$ is almost malnormal, then for every $\epsilon > 0$ there exists $D(\epsilon) > 0$ such that 
$$
{\rm diam} \left( \Nc_\epsilon(\Cc) \cap 
\Nc_\epsilon(\alpha \Cc)\right) 
\leq D(\epsilon)
$$
for all $\alpha \in \Gamma - \Gamma_0$. 
\label{prop:diam_flat_intersection}
\end{proposition} 

\begin{proof} 
Hersonsky--Paulin~\cite[Proposition 2.6]{HP2010} established this fact for $\CAT(-1)$ spaces and the same proof works for $\CAT(0)$ spaces. 
\end{proof}

The next two observations involve positive-valued polynomials. 

\begin{observation}\label{obs:poly_observation} If $Q \colon \Rb \rightarrow (0,\infty)$ is a positive-valued polynomial and $R \geq 0$, then 
$$
0 < \inf_{\abs{s-t} \leq R } \frac{Q(s)}{Q(t)} \leq \sup_{\abs{s-t} \leq R } \frac{Q(s)}{Q(t)} < +\infty.
$$
\end{observation} 

\begin{proof} Pick sequences $(s_n),(t_n)$
  with $\sup_{n \geq 1} \abs{s_n-t_n} 
  \leq R$ and 
$$
\lim_{n \rightarrow \infty} \frac{Q(s_n)}{Q(t_n)}= \sup_{\abs{s-t} \leq R } \frac{Q(s)}{Q(t)}.
$$
Passing to a subsequence we can suppose that $s:=\lim_{n \rightarrow \infty} s_n$ and $t : = \lim_{n \rightarrow \infty} t_n$ exist in $\Rb \cup \{\pm \infty\}$. If $s=\pm \infty$, then $t = s$ and since $Q$ is a polynomial 
$$
\lim_{n \rightarrow \infty} \frac{Q(s_n)}{Q(t_n)}= 1.
$$
Otherwise, $s,t \in \Rb$ and 
$$
\lim_{n \rightarrow \infty} \frac{Q(s_n)}{Q(t_n)}= \frac{Q(s)}{Q(t)}.
$$
In either case, we have 
$$
\sup_{\abs{s-t} \leq R } \frac{Q(s)}{Q(t)}=\lim_{n \rightarrow \infty} \frac{Q(s_n)}{Q(t_n)} < +\infty.
$$

The proof of the other inequality is exactly the same. 
\end{proof}

\begin{observation}\label{obs:all annuli are about Q times exp}
Suppose there exist $N_0 \in \Nb$  and a positive-valued polynomial $Q \colon \Rb \rightarrow (0,\infty)$ such that 
$$
\{ \gamma \in \Gamma_0 : n \leq \dist(o,\gamma o) < n + N_0\} \asymp Q(n)e^{\delta(\Gamma_0)n}. 
$$
Then for any $N_1, N_2 \in \Nb$ with $N_2 - N_1 \geq N_0$,
$$
\{ \gamma \in \Gamma_0 : n-N_1 \leq \dist(o,\gamma o) < n + N_2\} \asymp Q(n)e^{\delta(\Gamma_0)n}
$$
(with  implicit constants depending on $N_1, N_2$).
\end{observation}

\begin{proof} It suffices to consider $n > N_1$. Let $m : = \left\lfloor \frac{N_2-N_1}{N_0} \right\rfloor$. Then 
$$
\bigcup_{k=0}^{m-1} [n-N_1+kN_0, n-N_1 + (k+1) N_0)  \subset [n-N_1, n+N_2) 
$$
and 
$$
[n-N_1, n+N_2) \subset \bigcup_{k=0}^{m} [n-N_1+kN_0, n-N_1 + (k+1) N_0).
$$
Hence, 
$$
P_{m-1}(n) e^{\delta(\Gamma_0)n} \lesssim \{ \gamma \in \Gamma_0 : n-N_1 \leq \dist(o,\gamma o) < n + N_2\} \lesssim P_{m}(n) e^{\delta(\Gamma_0)n}
$$
where 
$$
P_j(n) = \sum_{k=0}^{j} Q(n-N_1+kN_0)e^{\delta(\Gamma_0)(-N_1+kN_0)} .
$$
Now the result follows from Observation~\ref{obs:poly_observation}.
\end{proof}

Next we show that if $\Cc$ satisfies Assumptions~\eqref{item:Morse
assumption} and~\eqref{item:parallel line assumption} in
Theorem~\ref{thm:main},   then any geodesic which is close to $\Cc$ for a
long time is very close to $\Cc$ for a slightly shorter time, see
Figure~\ref{fig:close for a very long time means very close for a long
time}.

\begin{figure}[h!]
  \centering
%
%
%

\begin{tikzpicture}[scale=1]
  \draw (0,-1)--(14,-1)--(10,1) node[above right] {$\Cc$} --(4,1)--cycle;

  \draw (.3,3) to[out=-30, in=140] (2,1) coordinate (ep1) to[out=-40, in=180] (4,.5) coordinate (r1)--
  (10,.5) coordinate (r2) to[out=0,in=-130] (12,1) coordinate (ep2)
  to[out=40, in=200](13,2);

  \draw[thick, dashed, black] (ep1)-- 
  node[midway, left] {$r$} 
  (2,-.3);
  \draw[thick, dashed, black] (ep2)-- 
  node[midway, right] {$r$} 
  (12,-.3);
  \draw[thick, dashed, black] (r1)-- node[midway, left] {$\epsilon$} (4,-.3); 
  \draw[thick, dashed, black] (r2)-- 
 node[midway, right] {$\epsilon$} 
  (10,-.3); 
  \draw[decorate,decoration={brace,amplitude=5pt}] 
  (4,-.3)--
  node[midway, below, yshift=-.5em] {$C$}
  (2,-.3);
  \draw[decorate,decoration={brace,amplitude=5pt}] 
  (12,-.3)-- 
  node[midway, below, yshift=-.5em] {$C$}
  (10,-.3);
\end{tikzpicture}

  \caption{
    A visual depiction of Proposition~\ref{prop:close for a very
    long time means very close for a long time}, in which we see geodesics
    which are $r$-close to $\Cc$ for sufficiently long time will be
    $\epsilon$-close to $\Cc$ for a slightly shorter time.
  }
  \label{fig:close for a very long time means very close for a long time}
\end{figure}

\begin{proposition}\label{prop:close for a very long time means very close for a long time} Assume, in addition, that $\Cc$ is a Morse subset and $\Cc$ contains all geodesic lines in $X$ which are parallel to a geodesic line in $\Cc$. For any 
  $r > \epsilon > 0$ there exists $C=C(r,\epsilon)>0$ such that: if $\ell
  \colon [a,b] \rightarrow X$ is a geodesic segment with 
$$
\ell([a,b]) \subset 
\Nc_{r}(\Cc),
$$
then 
$$
\ell([a+C,b-C]) \subset 
\Nc_{\epsilon}(\Cc).
$$
\end{proposition} 

\begin{proof} Suppose no such $C$ exists. Then for every $n \geq 1$ there
  exists a geodesic segment $\ell_n \colon [a_n,b_n] \rightarrow X$ such that 
$$
\ell_n([a_n,b_n]) \subset \Nc_{r}(\Cc) 
\quad \text{and} \quad \ell_n([a_n+n, b_n-n]) \not\subset 
\Nc_{\epsilon}( \Cc).
$$
For each $n \in \Nb$, fix $x_n \in \ell_n([a_n+n, b_n-n]) \smallsetminus 
\Nc_{\epsilon}( \Cc)$. 

Reparameterizing each $\ell_n$, we can assume that $\ell_n(0)=x_n$. Since $\Gamma_0$ acts cocompactly on $\Cc$ and $x_n \in 
\Nc_{r}(\Cc)$, 
after translating by $\Gamma_0$ and passing to a subsequence, we can suppose that $x_n \rightarrow x \in X$. Since 
$$
0  \in [a_n+n, b_n-n],
$$
we have $a_n \rightarrow -\infty$ and $b_n \rightarrow \infty$. Passing
to another subsequence, we can suppose that $\ell_n$ converges locally
uniformly to a geodesic line $\ell \colon \Rb \rightarrow X$. 

Then 
$$
\ell \subset 
\Nc_{r+1}(\Cc),
$$
which implies that $\ell$ is parallel to a geodesic line in $\Cc$ (see Lemma~\ref{lem:bounded implies really bounded}) and hence
$\ell \subset \Cc$. Since $x_n \rightarrow x=\ell(0) \in \Cc$
we see that $x_n \in
\Nc_{\epsilon}(\Cc)$ 
for $n$ sufficiently large. Hence, we have a contradiction. 
\end{proof} 

Recall that  $[\Cc]$ denotes the set of $\Gamma$-translates of $\Cc$. The
next result counts the elements of $[\Cc]$ which intersect a fixed annulus. 

\begin{proposition}  If $\delta(\Gamma_0) < \delta(\Gamma)$,   $\Cc$ is Morse, and there exists $N \in \Nb$ such that 
$$
\#\{ \gamma \in \Gamma : n \leq \dist(o, \gamma o) \leq n+N \} \asymp e^{\delta(\Gamma) n},
$$
then there exist $N' \in \Nb$ such that 
$$
 \#\{ \alpha\Cc \in [\Cc] : n <\dist(o, \alpha \Cc) \leq n+N'\}  \asymp e^{\delta(\Gamma)n}.
$$
\label{prop:counting_translates}
\end{proposition} 

We will use the following lemma from~\cite{HP2004}.

\begin{lemma}\cite[Lemma 3.3]{HP2004}\label{lem:HP let us count} For every
  $A>0$ and $\delta> \delta_0 > 0$ there exist $N' \in \Nb$ and $B > 0$ such that: if $\{b_n\}, \{c_n\} \subset \Rb_{>0}$ satisfy 
$$
b_n \leq Ae^{\delta n}, \quad c_n \leq Ae^{\delta_0 n}, \quad \text{and} \quad \sum_{k=0}^n b_k c_{n-k} \geq A^{-1} e^{\delta n}, 
$$
then 
$$
\sum_{k=1}^{N'} b_{n+k} \geq Be^{\delta n}.
$$
\end{lemma} 

\begin{proof}[Proof of Proposition~\ref{prop:counting_translates}]  Fix $r \in \Nb$ such that 
$$
\Cc \subset \Gamma_0 \cdot \Bc_r(o) \quad \text{and} \quad o \in \Nc_r(\Cc).
$$ 
(recall that $\Gamma_0$ acts cocompactly on $\Cc$).

Also,  
for $\alpha\in \Gamma$, we have 
$$
\alpha \Cc \subset \alpha \Gamma_0 \cdot \Bc_r(o).
$$
and so there exists $\gamma_{\alpha} \in \alpha\Gamma_0$ such that 
$$
\dist( \pi_{\alpha \Cc}(o), \gamma_\alpha o) < r. 
$$
We pick the elements $(\gamma_\alpha)_{ \alpha \in \Gamma}$ so that 
\begin{equation}\label{eqn:picking carefully}
\alpha \Cc = \beta \Cc \implies \gamma_\alpha = \gamma_\beta. 
\end{equation} 
 Thus, the mapping  
$\alpha\Cc\mapsto \gamma_\alpha$ is well-defined, and is also
finite-to-one by properness of the $\Gamma$-action on $X$.

Next fix $\sigma \geq 0$ satisfying Theorem~\ref{thm:equiv of Morse} for $\Cc$ and hence every $\Gamma$-translate of $\Cc$. Since $\alpha o \in \Nc_r(\alpha\Cc)$, Proposition~\ref{prop:projection to Morse subset} implies that 
\begin{equation*}
\abs{\dist(o,\alpha o) -  \dist(o, \pi_{\alpha \Cc}(o)) - \dist(\pi_{\alpha
\Cc}(o), \alpha o) } \leq  2\sigma+2r
\end{equation*}
and thus
\begin{equation}
  \label{e:approx_dist_gamma_o}
\abs{\dist(o,\alpha o) - \dist(o, \alpha \Cc) - \dist(\gamma_\alpha o,
\alpha o) } \leq  2\sigma+3r.
\end{equation}

Let 
$$
A_n : = \{ \alpha \in \Gamma : n \leq \dist(o, \alpha o) \leq n+N \}.
$$
Then, by hypothesis,
$$
a_n:=\#A_n\asymp e^{\delta(\Gamma) n}.
$$
Next let 
\begin{align*}
B_n & : = \{\alpha \Cc \in [\Cc] : n < \dist(o, \alpha \Cc) \leq n+1\}, \text{ and} \\
C_n & : = \{ \gamma \in \Gamma_0 : n - 2\sigma - r  -1< \dist(o,\gamma o) < n+N + 2\sigma     +3r
\}. 
\end{align*}
Also let 
\begin{align*}
b_n & : = \# B_n \quad \text{ and} \quad c_n : = \# C_n. 
\end{align*}

\medskip

\noindent \textbf{Claim: } $a_n \leq \sum_{k=0}^{n+N+r-1} b_kc_{n-k}$.

\medskip

\noindent \emph{Proof of Claim:} Equation~\eqref{eqn:picking carefully} implies that the map 
\begin{equation}\label{eqn:Gamma into [Cc] and Gamma0}
\alpha \in \Gamma \mapsto (\alpha \Cc, \alpha^{-1}\gamma_{\alpha}) \in [\Cc] \times \Gamma_0
\end{equation} 
is injective. Furthermore, if $\alpha \in A_n$, then Equation~\eqref{e:approx_dist_gamma_o} implies that
$$
n- 2\sigma 
-3r
< \dist(o, \alpha\Cc) + \dist(\gamma_{\alpha} o, \alpha o) \leq n+N+2\sigma 
+3r. 
$$
Thus, if $\alpha \Cc \in
B_k$, then $\alpha^{-1} \gamma_{\alpha} \in C_{n-k}$. Furthermore,
$$
\dist(o,\alpha \Cc) \leq \dist(o,\alpha o) +r< n+N+r
$$
and so $k \leq n+N+r-1$. Thus, the map in Equation~\eqref{eqn:Gamma into [Cc] and Gamma0} provides an injection
$$
A_n \hookrightarrow \bigcup_{k=0}^{n+N+r-1} B_k \times C_{n-k},
$$
which implies the claim. \hfill $\blacktriangleleft$

\medskip

Thus, since $\alpha \Cc = \gamma_\alpha \Cc$ and 
$$
\dist(o, \alpha \Cc) \geq \dist(o, \gamma_\alpha o) - r, 
$$
we have 
$$
b_n \leq \#\{ \gamma \in \Gamma : \dist(o,\gamma o) \leq n+1+r\} \lesssim e^{\delta(\Gamma)n}.
$$
Furthermore, if $\delta(\Gamma_0)< \delta_0 < \delta(\Gamma)$, then by the definition of critical exponent
$$
c_n \lesssim e^{\delta_0 n}.
$$
Finally, by the claim,
$$
\sum_{k=0}^{n} b_kc_{n-k} \geq a_{n-N-r+1} \gtrsim e^{\delta(\Gamma)n}. 
$$
Now the proposition follows from Lemma~\ref{lem:HP let us count}. 
\end{proof}

\section{Shadows of subspaces}
\label{s:shadow_lemma}

Let $X$, $\Gamma$, $\Cc$, $\Gamma_0$, $T_0$, and $Q$ be as in
Theorem~\ref{thm:main}.  Fix a basepoint $o \in \Cc$ and let $\mu$ be the unique Patterson--Sullivan measure for $\Gamma$ of dimension $\delta(\Gamma)$, see Theorem~\ref{thm:PS ergodic case}. 

In this section we estimate the $\mu$-measure of the shadows $\Sc_{T,\epsilon}(\Cc)$ introduced in Definition~\ref{def:shadows}.

\begin{theorem}[Subset Shadow Lemma]
  \label{thm:shadowing_flats} 
  For any $\epsilon > 0$ there exists $C > 1$ such that: if $T \geq 0$ and $\alpha \in \Gamma$, then 
  \[
\frac{1}{C} Q(T) e^{-\delta(\Gamma) (\dist(o,\alpha \Cc)+T)+\delta(\Gamma_0)T}  \leq \mu(\Sc_{T,\epsilon}(\alpha \Cc)) \leq C Q(T) e^{-\delta(\Gamma) (\dist(o,\alpha \Cc)+T)+\delta(\Gamma_0)T}.   \]

\end{theorem}

The rest of the section is devoted to the proof of Theorem~\ref{thm:shadowing_flats}. We start by fixing some constants. Fix $\sigma \geq 0$ as in Theorem~\ref{thm:equiv of Morse} and then fix 
\begin{equation}\label{eqn: R > sigma + 2 epsilon}
R >\sigma+\epsilon
\end{equation}
large enough so that any number in $[R,\infty)$ satisfies the Shadow Lemma (Proposition~\ref{prop:shadow lemma}). Also fix $r > 0$ such that 
\begin{equation}\label{eqn:covering}
\Gamma_0 \cdot \Bc_r(o) \supset \Cc.
\end{equation} 

The following approximation result will be useful for both the lower and upper bound in Theorem~\ref{thm:shadowing_flats}.

\begin{lemma}\label{lem:existence of gamma_alpha} If $\alpha \in \Gamma$, then there exists $\gamma_\alpha \in \alpha\Gamma_0$ such that 
$$
\dist( \gamma_\alpha o, \pi_{\alpha \Cc}(o)) \leq r.
$$
Moreover, 
$$
\alpha \Cc \subset \bigcup_{\gamma \in \Gamma_0} \Bc_r(\gamma_\alpha \gamma o)
$$
and 
$$
\abs{\dist(o,\gamma_\alpha \gamma o) - (\dist(o,\gamma_\alpha o) + \dist(o, \gamma o))} \leq 2\sigma +2r
$$
for all $\gamma \in \Gamma_0$ 
\end{lemma} 

\begin{proof} Notice that $\alpha^{-1}  \pi_{\alpha \Cc}(o) \in \Cc$. Then by Equation~\eqref{eqn:covering}, there exists $\gamma \in \Gamma_0$ with $\dist(\gamma o, \alpha^{-1}  \pi_{\alpha \Cc}(o)) \leq r$. Let $\gamma_\alpha : = \alpha \gamma$. Then 
$$
\dist( \gamma_\alpha o, \pi_{\alpha \Cc}(o)) \leq r.
$$
From Equation~\eqref{eqn:covering}, 
$$
\alpha \Cc \subset \alpha\Gamma_0 \cdot \Bc_r(o) = \gamma_\alpha\Gamma_0 \cdot \Bc_r(o) =\bigcup_{\gamma \in \Gamma_0} \Bc_r(\gamma_\alpha \gamma o)
$$
and by Proposition~\ref{prop:projection to Morse subset},
\begin{align*}
& \abs{\dist(o,\gamma_\alpha \gamma o) - (\dist(o,\gamma_\alpha o) + \dist(o, \gamma o))} \\
& \quad \leq 2r +\abs{\dist(o,\gamma_\alpha \gamma o) - (\dist(o,\pi_{\alpha \Cc}(o)) + \dist(\pi_{\alpha \Cc}(o), \gamma_\alpha\gamma o))} \leq 2\sigma + 2r. \qedhere
\end{align*}

\end{proof}

\subsubsection{The upper bound} 

We start by relating the set $\Sc_{T,\epsilon}(\alpha \Cc)$ to standard shadows. 

\begin{lemma}\label{lem:S covered by shadows} If $\alpha \in \Gamma$, then 
$$
\Sc_{T,\epsilon}(\alpha  \Cc)
\subset \bigcup \left\{ \Oc_R(o,x) : x \in \alpha \Cc, \, \dist(x,\pi_{\alpha \Cc}(o)) = \max(0,T-2 \sigma-4 \epsilon)\right\}. 
$$
\end{lemma}

\begin{proof} 
Fix $\eta \in \Sc_{T,\epsilon}(\alpha  \Cc)$. Let $\ell_\eta \colon[ 0,\infty) \rightarrow X$ be the geodesic ray starting at $o$ and limiting to $\eta$. Then let
$$
(a,b) : = \left\{ t \geq 0 : \ell_\eta(t) \in 
\Nc_\epsilon(\alpha \Cc)
\right\}
$$
(which is indeed an interval by Lemma~\ref{lem:following_flats}). Then $b-a \geq T$. 

Proposition~\ref{prop:projection to Morse subset}(1) implies that $\ell_\eta|_{[0,b]}$ intersects 
$ \Bc_{\sigma+\epsilon}(\pi_{\alpha \Cc}(o))$. 
Let 
$$
t_\star : = \inf\{ t \in [0,b] : \ell_\eta(t) \in \Bc_{\sigma+\epsilon}(\pi_{\alpha \Cc}(o))\}.
$$
Notice that 
$$
\dist(o, \alpha \Cc) = \dist(o,\pi_{\alpha \Cc}(o))  \geq t_\star - \sigma - \epsilon
$$
and 
$$
\dist(o, \alpha \Cc) \leq \dist(o,\ell_\eta(a)) + \epsilon = a + \epsilon.
$$
Hence, 
$$
t_\star \leq a + \sigma+2 \epsilon.
$$

Next let $\ell \colon [0,b''] \rightarrow X$ denote the geodesic joining
$\pi_{\alpha \Cc}(o)$ to $y:=\pi_{\alpha  \Cc}(\ell_\eta(b))$. Then by
Lemma~\ref{lem:HD_between_geod}, 
\begin{equation}\label{eqn:dist between ell eta and ell}
\dist^{\rm Haus}(\ell_\eta|_{[t_\star, b]}, \ell) \leq \max\big(\dist(\ell_\eta(t_\star), \, \ell(0)),\dist(\ell_\eta(b), \ell(b''))\big)\leq \sigma+\epsilon.
\end{equation}
Also, 
\begin{align*}
b''& = \dist(\pi_{\alpha \Cc}(o),y) \geq \dist(\ell_\eta(t_\star), \ell_\eta(b)) - \sigma-2\epsilon= b-t_\star -\sigma -2\epsilon \\
& \geq b-a - 2 \sigma - 4\epsilon \geq T-2\sigma - 4\epsilon. 
\end{align*}
Hence, $x := \ell(\max(0,T - 2 \sigma - 4\epsilon))$ is well-defined.
Then Equation~\eqref{eqn:dist between ell eta and ell} and the assumption that $R > \sigma+\epsilon$, see Equation~\eqref{eqn: R > sigma + 2 epsilon}, imply that
\begin{equation*}
\eta \in \Oc_{\sigma+\epsilon}(o,x) \subset \Oc_R(o,x). \qedhere
\end{equation*}

\end{proof}

For $T \geq 0$, let 
$$
f(T) : = \max(0,T - 2 \sigma - 4\epsilon).
$$
Then let
\begin{align*}
A_{T} : =\left\{ \gamma \in \Gamma_0 : \dist(o,\gamma o) \in [f(T)-2r, f(T)+2r]\right\}.
\end{align*}
For any $x \in \alpha \Cc$ with 
$$
\dist(x,\pi_{\alpha \Cc}(o)) = f(T),
$$
Lemma~\ref{lem:existence of gamma_alpha} implies that there exists $\gamma \in \Gamma_0$ with 
$$
\dist(\gamma_\alpha \gamma o,x) < r
$$
and hence
$$
\dist(\gamma o, o) = \dist(\gamma_\alpha \gamma o, \gamma_\alpha o) \in  \dist(x, \pi_{\alpha \Cc}(o))+[-2r,2r]=[f(T)-2r,f(T)+2r].
$$

Thus,
$$
\bigcup\{ \Oc_R(o,x) : x \in \alpha \Cc, \, \dist(\pi_{\alpha \Cc}(o),x) = f(T)\} \subset \bigcup_{\gamma \in A_T} \Oc_{R+r}(o,\gamma_\alpha \gamma o).
$$
Then by Lemma~\ref{lem:S covered by shadows}, we have 
$$
\mu(\Sc_{T,\epsilon}(\alpha  \Cc) ) \leq \sum_{\gamma \in A_{T} } \mu( \Oc_{R+r}(o,\gamma_\alpha \gamma o) ).
$$
By Proposition~\ref{prop:projection to Morse subset}, 
$$
 e^{-\delta(\Gamma)\dist(o,\gamma_\alpha \gamma o) } \gtrsim e^{-\delta(\Gamma) (\dist(o,\alpha \Cc)+T)} 
 $$
 when $\gamma \in A_T$. Finally, by the Shadow Lemma (Proposition~\ref{prop:shadow lemma}) and Proposition~\ref{obs:all annuli are about Q times exp}, 
\begin{align*}
\mu(\Sc_{T,\epsilon}(\alpha  \Cc) ) &  \lesssim \sum_{\gamma \in A_{T} } e^{-\delta(\Gamma)\dist(o,\gamma_\alpha \gamma o) } \\
& \lesssim Q(T) e^{-\delta(\Gamma) (\dist(o,\alpha \Cc)+T)} e^{\delta(\Gamma_0)T}. 
\end{align*}

\subsubsection{The lower bound} By Proposition~\ref{prop:close for a very long time means very close for a long time}, there exists $C > 0$ such that: if $\ell \colon [a,b] \rightarrow X$ is a geodesic segment, then 
\begin{equation}\label{eqn:using long tracking => close}
\ell([a,b]) \subset 
\Nc_{\sigma+2R+1}(\Cc) \implies \ell([a+C,b-C]) \subset \Nc_{\epsilon}(\Cc).
\end{equation}

\begin{lemma}\label{lem:S contains some shadows} 
$$
\Sc_{T,\epsilon}(\alpha \Cc) \supset \bigcup \left\{ \Oc_R(o,\gamma_\alpha \gamma o) : \gamma \in \Gamma_0, \, \dist(o,\gamma o) \geq T+2C+R+r+\sigma \right\}
$$
(recall that $\gamma_\alpha$ is defined in Lemma~\ref{lem:existence of gamma_alpha}). 
\end{lemma} 

\begin{proof} Suppose that $\eta \in\Oc_R(o,\gamma_\alpha \gamma o) $ where $\gamma \in \Gamma_0$ and 
$$
\dist(o, \gamma o) \geq T+2C+R+r+\sigma. 
$$
Let $\ell_\eta :[ 0,\infty) \rightarrow X$ be the geodesic ray
starting at $o$ and limiting to $\eta$. Fix $t_0 \geq 0$ such that
$\dist(\ell_\eta(t_0), \gamma_\alpha \gamma o) < R$. Let $\ell \colon [0,b] \rightarrow X$ be the geodesic segment joining $o$ to $\gamma_\alpha\gamma o$. Notice that 
$$
\abs{t_0-b} <R. 
$$

Recall that $o \in \Cc$ and $\gamma_\alpha \in \alpha \Gamma_0$. Thus, $\gamma_\alpha\gamma o \in \alpha \Cc$. Then Proposition~\ref{prop:projection to Morse subset} implies that there exists $t_1 \in [0,b]$ such that 
$$
\dist(\ell(t_1),\pi_{\alpha \Cc}(o)) \leq \sigma. 
$$
By Lemma~\ref{lem:HD_between_geod}, 
$$
\dist^{\rm Haus}(\ell, \ell_\eta|_{[0,t_0]}) \leq \dist(\ell(b), \ell_\eta(t_0)) < R. 
$$
Hence, 
$$
\dist(\ell(s), \ell_\eta(t_1)) < R
$$
for some $s \in [0,b]$. Since 
$$
R > \dist(\ell(s), \ell_\eta(t_1)) \geq \abs{ \dist(\ell(s), o) - \dist(o,\ell_\eta(t_1))} = \abs{s-t_1},
$$
then 
$$
\dist(\ell(t_1), \ell_\eta(t_1)) \leq\dist(\ell(s), \ell_\eta(t_1))+\abs{s-t_1} < 2R. 
$$
Hence, $\ell_\eta(t_1) \in \Bc_{\sigma+2R+1}(\pi_{\alpha \Cc}(o))$. Since $\alpha \Cc$ is convex, Lemma~\ref{lem:HD_between_geod} implies that
$$
\ell_\eta([t_1, t_0]) \subset \Nc_{\sigma+2R+1}(\alpha \Cc). 
$$
Then Equation~\eqref{eqn:using long tracking => close} implies that
$$
\ell_\eta([t_1+C, t_0-C]) \subset 
\Nc_{\epsilon}(\alpha \Cc). 
$$
Now notice that 
\begin{align*}
b-t_1 & = \dist( \gamma_\alpha \gamma o, \ell(t_1)) \geq  \dist( \gamma_\alpha \gamma o,\pi_{\alpha \Cc}(o))-\sigma  \geq  \dist( \gamma_\alpha \gamma o,\gamma_\alpha o)- \dist( \gamma_\alpha o,\pi_{\alpha \Cc}(o))-\sigma \\
& \geq \dist(\gamma o,o)-r -\sigma \geq T+2C+R.
\end{align*}
Hence,
\begin{align*}
(t_0-C)-(t_1+C) \geq (b-t_1) - 2C-R \geq T  
\end{align*}
and so $\eta \in \Sc_{T,\epsilon}(\alpha \Cc)$. 
\end{proof} 

Let 
\begin{align*}
B_{T} :&  =\left\{ \gamma \in \Gamma_0 : \dist(o,\gamma o) \in [T+2C+R+r+\sigma, T+2C+R+r+\sigma+T_0]\right\}.
\end{align*}
Fix a maximal subset $B_T' \subset B_T$ such that: if $\gamma_1, \gamma_2 \in B_T'$ are distinct, then 
$$
\dist(\gamma_\alpha \gamma_1 o, \gamma_\alpha \gamma_2 o) \geq T_0 + 4r +4\sigma +4R.
$$

\begin{lemma}\label{lem:disjoint shadows in an annuli} If $\gamma_1, \gamma_2 \in B_T'$ are distinct, then
$$
 \Oc_{R}(o,\gamma_\alpha \gamma_1 o) \cap  \Oc_{R}(o,\gamma_\alpha \gamma_2 o) = \emptyset.
 $$
\end{lemma} 

\begin{proof} Suppose $\gamma_1, \gamma_2 \in B_T'$ and 
$$
\eta \in  \Oc_{R}(o,\gamma_\alpha \gamma_1 o) \cap  \Oc_{R}(o,\gamma_\alpha \gamma_2 o). 
 $$
 Let $\ell_\eta \colon[ 0,\infty) \rightarrow X$ be the geodesic ray starting at $o$ and limiting to $\eta$.  Then there exists $t_1, t_2 \geq 0$ such that 
 $$
 \ell_\eta(t_1) \in \Bc_R(\gamma_\alpha \gamma_1 o) \quad \text{and} \quad \ell_\eta(t_2) \in \Bc_R(\gamma_\alpha \gamma_2 o).
 $$
 By Lemma~\ref{lem:existence of gamma_alpha},
 $$
\abs{ \dist(o, \gamma_\alpha \gamma_i o) - \dist(o, \gamma_\alpha o) - \dist(\gamma_\alpha o, \gamma_\alpha \gamma_i o) } \leq 2r + 2\sigma
 $$ 
and so
 \begin{align*}
 \abs{t_1-t_2} &<  \abs{\dist(o,\gamma_\alpha\gamma_1o)-\dist(o,\gamma_\alpha\gamma_2o)}+2R \leq \abs{\dist(o,\gamma_1o)-\dist(o,\gamma_2o)}+2R + 4r+4\sigma \\
 & \leq  T_0 + 4r +4\sigma +2R.
 \end{align*}
 Thus, 
 $$
 \dist(\gamma_\alpha \gamma_1 o, \gamma_\alpha \gamma_2 o) \leq  \abs{t_1-t_2} +2R < T_0 + 4r +4\sigma +4R.
 $$
 We conclude $\gamma_1 = \gamma_2$.
\end{proof}

Then by Lemmas~\ref{lem:S contains some shadows} and ~\ref{lem:disjoint
shadows in an annuli}, we have 
$$
\mu(\Sc_{T,\epsilon}(\alpha  \Cc) ) \geq \sum_{\gamma \in B_{T}' } \mu( \Oc_{R}(o,\gamma_\alpha \gamma o) ).
$$
Since $\Gamma$ is discrete, we have $\# B_T' \gtrsim \#B_T$. By Proposition~\ref{prop:projection to Morse subset}, 
$$
 e^{-\delta(\Gamma)\dist(o,\gamma_\alpha \gamma o) } \gtrsim e^{-\delta(\Gamma) (\dist(o,\alpha \Cc)+T)} 
 $$
when $\gamma \in B_T'$. Then by the Shadow Lemma (Proposition~\ref{prop:shadow lemma}) and Proposition~\ref{obs:all annuli are about Q times exp}, 
\begin{align*}
\mu(\Sc_{T,\epsilon}(\alpha  \Cc) ) &  \gtrsim \sum_{\gamma \in B_{T}' } e^{-\delta(\Gamma)\dist(o,\gamma_\alpha \gamma o) } \gtrsim e^{-\delta(\Gamma) (\dist(o,\alpha \Cc)+T)} \#B_T' \\
& \gtrsim e^{-\delta(\Gamma) (\dist(o,\alpha \Cc)+T)} \#B_T \gtrsim Q(T) e^{-\delta(\Gamma) (\dist(o,\alpha \Cc)+T)} e^{\delta(\Gamma_0)T}. 
\end{align*}

This completes the proof of Theorem~\ref{thm:shadowing_flats}.

\section{Proof of Theorem~\ref{thm:khinchin}}
\label{s:main}

In this section we prove Theorem~\ref{thm:khinchin}. For the rest of the section let $X$, $\Gamma$, $\Cc$, $\Gamma_0$, $T_0$, and $Q$ be as in
Theorem~\ref{thm:main}. Fix a basepoint $o \in \Cc$ and let $\mu$ be the
unique Patterson--Sullivan measure for $\Gamma$ of dimension
$\delta(\Gamma)$, see Theorem~\ref{thm:PS ergodic case}.

For the readers convenience we recall the statement of Theorem~\ref{thm:khinchin} and the notation used in the statement.    For any $\alpha\in\Gamma$ and $T,\epsilon>0$, we defined
  \[
       \Sc_{T,\epsilon}(\alpha \Cc)
    = \{\eta\in\partial X \mid \ell_{\eta}([a,b])\subset \mathcal
      N_\epsilon (\alpha\Cc) \text{ for some } a,b\in[0,\infty) \text{ with }
    b-a\geq T\}.
  \]
Given a function $\phi : [0, \infty) \rightarrow [0,\infty)$, we denoted 
\begin{align*}
\phi\Sc_{T,\epsilon}(\alpha\Cc) & = \Sc_{T+\phi(\dist(o,\alpha\Cc)),\epsilon}(\alpha\Cc), \\
  \Theta^\phi_{T,\epsilon} & =\{\xi \in\partial X\mid \xi\in \phi\Sc_{T,\epsilon}(\alpha\Cc)
    \textrm{ for infinitely many
  }\alpha \Cc\in [\Cc] \}, \text{ and} \\
    K^\phi & = \sum_{n\in\mathbb N}
e^{-(\delta(\Gamma)-\delta(\Gamma_0))\phi (n)}
    Q(\phi(n)).
\end{align*}

\begin{theorem}[Khinchin-type theorem]
Given $\epsilon>0$ there exists $T_0 \geq 0$ such that for any $\phi
\colon 
[0,\infty) \rightarrow [0,\infty)$  
slowly varying and any $T\geq T_0$, we have the following dichotomy:
  \begin{enumerate}
    \item If $K^\phi<\infty$, then $\mu(\Theta^\phi_{T,\epsilon})=0$.
    \item If $K^\phi=\infty$, then $\mu(\Theta^\phi_{T,\epsilon})=1$.
  \end{enumerate}
  \label{thm:khinchin_1}
\end{theorem}

To prove Theorem~\ref{thm:khinchin_1} we use the  Borel--Cantelli Lemma and its converse (see
\cite[Section II]{Lamperti} for a proof). More precisely, we will show that when $\phi$ is slowly varying and $T$ is
sufficiently large, a certain sequence of sets 
$(U^\phi_{n,T,\epsilon})$ 
satisfy the hypotheses of the Borel--Cantelli Lemma.

\begin{lemma}[Borel--Cantelli]
  Let $(Y,\nu)$ be a probability space and $(Y_n)_{n \in \Nb}$ a sequence of
  measureable subsets of $Y$. Then:
  \begin{enumerate}
    \item If $\sum_{n=1}^\infty \nu(Y_n)<\infty$, then
      $\nu(\limsup_{n\to\infty} Y_n)=0$.
    \item If $\sum_{n=1}^\infty \nu(Y_n)=\infty$ and there exists a constant
      $c$ such that $\nu(Y_n\cap Y_m)\leq c\nu(Y_n)\nu(Y_m)$ for all $n\neq
      m$ (in other words, the family $(Y_n)$ is {\em quasi-independent}), 
      then $\nu(\limsup_{n\to\infty} Y_n)>0$.
  \end{enumerate}
  \label{lem:borel_cantelli}
\end{lemma}

Define
\[
  \Ac_n:=
  \{{\alpha\Cc}\in[\Cc] \mid \dist(o,{\alpha\Cc})\in
  [n,n+1)\},
\]
and 
\[
U^\phi_{n,T,\epsilon}:=\bigcup_{{\alpha\Cc}\in \Ac_n} \phi
  {\Sc}_{T,\epsilon}(\alpha\Cc).
\]
Then 
\[
\Theta^\phi_{T,\epsilon}:=\limsup_{n\to\infty} U^\phi_{n,T,\epsilon}
  =\bigcap_{n\in\mathbb N}\bigcup_{m\geq n} U^\phi_{m,T,\epsilon}.
\]

\subsection{Measure estimates} 
The goal of this section is to prove the following. 

\begin{proposition}
Assume $\epsilon>0$,  $\phi \colon [0,\infty) \rightarrow [0,\infty)$ 
is slowly varying, and $T > 0$  is sufficiently large (depending only on $\epsilon$), then 
$$
\sum_{n=1}^\infty \mu(U^\phi_{n,T,\epsilon}) = \infty \quad \text{if and only if} \quad K^\phi = \infty.
$$
  \label{prop:sum condition in BC}
\end{proposition}

The proof of Proposition~\ref{prop:sum condition in BC} above follows from
Lemma~\ref{lem:horoball_shadows}, Proposition~\ref{prop:disjointness}, and
Lemma~\ref{lem:quasi_ind_lemma1}.

\begin{lemma} Assume $\epsilon>0$, $T \geq 0$, and $\phi \colon [0,\infty)
  \rightarrow [0,\infty)$ is slowly varying.
  If ${\alpha\Cc}\in \Ac_n$, then
  \[
    \mu(\phi {\Sc}_{T,\epsilon}(\alpha\Cc))\asymp
    e^{-\delta(\Gamma) n}
    e^{(\delta(\Gamma_0)-\delta(\Gamma))
    \phi(n)} Q(\phi(n))
  \]
  with implicit constants independent of $n$ and ${\alpha}\Cc\in\Ac_n$. 
  \label{lem:horoball_shadows}
\end{lemma}

\begin{proof}
  By the   Subset Shadow Lemma  (Theorem~\ref{thm:shadowing_flats}),
  \[
    \mu({\Sc}_{t,\epsilon}(\alpha\Cc))\asymp e^{-\delta(\Gamma) (\dist(o,{\alpha\Cc})+t)}e^{\delta(\Gamma_0)
    t}Q(t)
  \]
  with implicit constants depending only on $\epsilon$. Hence,
    \[
    \mu(\phi {\Sc}_{T,\epsilon}(\alpha\Cc))\asymp e^{-\delta(\Gamma) \dist(o,{\alpha\Cc})}
    e^{(\delta(\Gamma_0)-\delta(\Gamma))\phi(\dist(o,{\alpha\Cc}))} Q(\phi(\dist(o,{\alpha\Cc}))+T)
  \]
  with implicit constants depending only on $\epsilon,T$. Since
  ${\alpha\Cc}\in\Ac_n$, 
  $\phi$ is slowly varying, and $Q$ is a positive-valued
  polynomial,  the conclusion follows   from Observation~\ref{obs:poly_observation}.
\end{proof}

\begin{proposition}
  For all   $\epsilon>0$  there exists $T_1>0$ such that: If 
  $T\geq T_1$,  and $n\in\mathbb N$, then
  the sets 
  $$
  \{{\Sc}_{T,\epsilon}(\alpha\Cc)\mid {\alpha\Cc}\in
  \Ac_n\}
  $$
  are pairwise disjoint.   More explicitly, we have 
    \[
      \mathcal S_{T,\epsilon}(\alpha_1\Cc)\cap\mathcal
      S_{T,\epsilon}(\alpha_2\Cc) = \emptyset
    \]
    if $\alpha_1\Cc,\alpha_2\Cc\in\Ac_n$ and $\alpha_1\Cc\neq\alpha_2\Cc$. 
  \label{prop:disjointness}
\end{proposition}

\begin{proof} 
  Let $\sigma \geq 0$ satisfy Theorem~\ref{thm:equiv of Morse} for $\Cc$ and let $D=D(\epsilon+\sigma) > 0$ satisfy Proposition~\ref{prop:diam_flat_intersection}. Then set
  $$
  T_1 : = D + 6\epsilon + 6\sigma+2. 
  $$

Fix $T \geq T_1$ and   $\alpha_1\Cc,\alpha_2\Cc\in\Ac_n$ with 
  $$
  \eta\in {\Sc}_{T,\epsilon}(\alpha_1\Cc)\cap {\Sc}_{T,\epsilon}(\alpha_2\Cc).
  $$
  Let $\ell_\eta \colon [ 0,\infty) \rightarrow X$ be the geodesic ray starting at $o$ and limiting to $\eta$. For $i=1,2$, let 
$$
(a_i,b_i) : = \left\{ t \geq 0 : \ell_\eta(t) \in 
\Nc_{\epsilon+\sigma}(\alpha_i \Cc)\right\} 
$$
(which is indeed an interval by Lemma~\ref{lem:following_flats}). Then 
$$
b_i  -a_i \geq T.
$$
  After possibly relabelling, we can suppose that  $a_1\leq a_2$. Then 
  $$
  [a_2, a_1 + T] \subset [a_1,a_1+T] \cap [a_2,a_2+T] \subset  [a_1,b_1] \cap [a_2,b_2]
  $$
  (where $[a_2,a_1+T]=\emptyset$ if $a_1 +T < a_2$).

Since $\alpha_1\Cc,\alpha_2\Cc\in\Ac_n$, we have 
$$
|\dist(o,\alpha_1\Cc)-\dist(o,\alpha_2\Cc)|\leq 1.
  $$
  Now Proposition~\ref{prop:projection to Morse subset}(2) implies that
  \begin{align*}
    \abs{a_1-a_2}&=\abs{d(o,\ell_\eta(a_1))-d(o,\ell_\eta(a_2))}\\
    &\leq
    \abs{d(\ell_\eta(a_1),\pi_{\alpha_1\Cc}(o))-d(\ell_\eta(a_2),\pi_{\alpha_2\Cc}(o))}+\abs{d(o,\alpha_1\Cc)-d(o,\alpha_2\Cc)}
    \\ & \leq 
    6\sigma+6\epsilon+
    1.
  \end{align*}
Hence, 
\begin{align*}
\diam  &  (\Nc_{\epsilon+\sigma}(\alpha_1\Cc)\cap \Nc_{\epsilon+\sigma}(\alpha_2\Cc))  \geq \diam \ell_\eta([a_2,a_1+T])  \\
    &  = a_1+T - a_2 \geq T - 6\sigma - 6\epsilon - 
    1
    > D(\epsilon+\sigma).
\end{align*} 
Thus, by the definition of $D(\epsilon+\sigma)$, we have $\alpha_1 \Cc = \alpha_2 \Cc$. 
\end{proof}

\begin{lemma} Assume $\epsilon>0$,  $\phi \colon [0,\infty) \rightarrow [0,\infty)$ 
  is slowly varying, 
  and $T_1 = T_1(\epsilon)> 0$ satisfies Proposition~\ref{prop:disjointness}.
 If $\alpha_0\Cc\in \Ac_n$ and $T \geq T_1$, then
  \begin{align*}
    \mu(U_{n,T,\epsilon}^\phi)&\asymp
\sum_{\alpha\Cc\in\Ac_n}\mu(\phi\Sc_{T,\epsilon}(\alpha\Cc))
 \asymp 
    \mu(\phi {\Sc}_{T,\epsilon}(\alpha_0\Cc))\#\Ac_n \\
&    \asymp \frac{\mu(\phi {\Sc}_{T,\epsilon}(\alpha_0\Cc))}{\mu({\Sc}_{T,\epsilon}(\alpha_0\Cc))} \asymp 
    e^{(\delta(\Gamma_0)-\delta(\Gamma))\phi(n)}Q(\phi(n)).
      \end{align*}
  with implicit constants independent of $n$ and ${\alpha}_0\Cc\in\Ac_n$. 
  \label{lem:quasi_ind_lemma1}
\end{lemma}

\begin{remark} Though it is not significant for the proof of Theorem~\ref{thm:khinchin_1}, an interesting observation is
that when $\phi \equiv 1$ we obtain $\mu(U_{n,T,\epsilon}^\phi)\asymp1$. \end{remark} 

\begin{proof}
Since $\phi\geq 0$ we have $\phi {\Sc}_{T,\epsilon}(\alpha\Cc)\subset
{\Sc}_{T,\epsilon}(\alpha\Cc)$.  Then by
Proposition~\ref{prop:disjointness} and Lemma~\ref{lem:horoball_shadows}, the sets $\{ \phi\Sc_{T,\epsilon}(\alpha \Cc)  : \alpha \in \Ac_n\}$ are disjoint and have coarsely equal measure. Hence, the first two coarse equalities hold.  The final two coarse equalities follow
  Lemma~\ref{lem:horoball_shadows} and
Proposition~\ref{prop:counting_translates}.
\end{proof}

\begin{proof}[Proof of Proposition~\ref{prop:sum condition in BC}] This follows immediately from the last coarse equality in Lemma~\ref{lem:quasi_ind_lemma1}.
\end{proof}

\subsection{Quasi-independence} 

The goal of this section is to prove the following.

\begin{proposition}[Quasi-independence] 
Assume $\epsilon>0$,  $\phi \colon [0,\infty) \rightarrow [0,\infty)$ is slowly varying, and $T > 0$  is sufficiently large (depending only on $\epsilon$), then there exists $c=c(\phi,\epsilon,T)> 0$ such that
  $$
  \mu(U^{\phi}_{m,T,\epsilon}\cap U^{\phi}_{n,T,\epsilon})\leq c\mu(U^{\phi}_{m,T,\epsilon})\mu(U^{\phi}_{n,T,\epsilon})
  $$
  for all $m > n \geq 1$. 
  \label{prop:quasi_independence}
\end{proposition}

Fix $\epsilon> 0$ and $\phi$ slowly varying. We need two lemmas. 

\begin{lemma}
There exist $C_1,T_2>0$ (both depending only on $\epsilon$) such that:  if $T\geq T_2$,
  $\alpha_1,\alpha_2\in\Gamma$, $\alpha_1 \Cc \neq \alpha_2 \Cc$, 
  $\dist(o,\alpha_1\Cc)\leq \dist(o,\alpha_2\Cc)$, and $$
  \phi {\Sc}_{T,\epsilon}(\alpha_1\Cc)\cap \phi {\Sc}_{T,\epsilon}(\alpha_2\Cc)\neq \emptyset,
  $$
  then 
  $$
  {\Sc}_{T,\epsilon}(\alpha_2\Cc)\subset \phi {\Sc}_{T-C_1,\epsilon+C_1}(\alpha_1\Cc). 
  $$
  \label{lem:nesting_shadows}
\end{lemma}

\begin{proof}
  Let $\sigma > 0$ satisfy Theorem~\ref{thm:equiv of Morse}  for $\Cc$ and let $D=D(\epsilon+\sigma) > 0$ satisfy Proposition~\ref{prop:diam_flat_intersection}. Then set
  $$
  T_2 : = D + 4\epsilon + 4\sigma. 
  $$

  Fix $T \geq T_2$, $\eta\in\phi\Sc_{T,\epsilon}(\alpha_1\Cc)\cap\phi\Sc_{T,\epsilon}(\alpha_2\Cc)$, and
  $\xi\in\Sc_{T,\epsilon}(\alpha_2\Cc)$.   Let $\ell_\eta,\ell_\xi \colon [ 0,\infty) \rightarrow X$ be the geodesic rays starting at $o$ and limiting to $\eta, \xi$ respectively. Then let
  $$
  (a_i,b_i) : = \left\{ t \geq 0 : \ell_\eta(t) \in 
  \Nc_{\epsilon+\sigma}(\alpha_i \Cc)\right\} 
  \quad \text{for} \quad i=1,2
  $$
  and 
  $$
  (a_2',b_2') : = \left\{ t \geq 0 : \ell_\xi(t) \in 
  \Nc_{\epsilon+\sigma}(\alpha_2 \Cc)\right\}  $$
  (which are indeed intervals by Lemma~\ref{lem:following_flats}). Then 
  $$
  b_i - a_i \geq T+\phi(\dist(o,\alpha_i\Cc))  \quad \text{for} \quad i=1,2.
  $$
  See Figure~\ref{fig:nesting_shadows}.

  \begin{figure}[h!]
    \centering
%
%
%

\begin{tikzpicture}[scale=1]

  \draw (0,.5)--(3,.5)--
  (4,2.5)
  --(1,2.5)--
  node[pos=.2, left] {$\alpha_1\Cc$}
  cycle;

  \draw (4,-2) --
  (7,-2)-- 
  node[pos=.15, right] {$\alpha_2\Cc$} 
  (8,0)--
  (5,0)--cycle;

  \draw (-1,.75)-- 
  coordinate[pos=.2] (a1) 
  coordinate[pos=.25] (t1) 
  coordinate[pos=1] (s1) 
  (3.25,.75) to[out=0,in=180]
  coordinate[pos=.1] (b1) 
  coordinate[pos=.9] (a2) 
  (4.4,-.75) to [out=0,in=-120]
  coordinate[pos=.75] (b2) 
  (9,.5) node[above right] {$\eta$};

  \draw (-1,.3)--
  coordinate[pos=.2] (a1') 
  (2.8,.3) to[out=0,in=180]
  coordinate[pos=.2] (b1')
  coordinate[pos=.9] (a2')
  (4,-1.3) to [out=0, in=90] 
  coordinate[pos=.75] (b2')
  (6,-3) node[below] {$\xi$};

  \draw (a1)+(.25,.2) coordinate (na1);
  \draw (b1)+(-.1,.2) coordinate (nb1);
  \draw (a2)+(.5,.4) coordinate (na2);
  \draw (b2)+(-.4,.1) coordinate (nb2);
  \draw (a2')+(.1,-.5) coordinate (na2');
  \draw (b2')+(-.5,-.1) coordinate (nb2');

  \draw [decorate,decoration={brace,amplitude=5pt}]
  (na1)--(nb1)
  node [midway,xshift=5pt,yshift=12pt] {\footnotesize $\geq T+\phi(d(o,\alpha_1
  \Cc))$}; \draw [decorate,decoration={brace,amplitude=5pt}]
  (na2)--(nb2)
  node [midway,xshift=5pt,yshift=15pt] {\footnotesize $\geq T+\phi(d(o,
  \alpha_2\Cc))$};
  \draw [decorate,decoration={brace,amplitude=5pt}]
  (nb2')--(na2')
  node [midway,xshift=-12pt,yshift=-12pt] {\footnotesize $\geq T$};

  \def\s{.05}
  \draw[fill] (a1') circle (\s) node[below left] {$a_1'$};
  \draw[fill] (b1') circle (\s) node[below left] {$b_1'$};
  \draw[fill] (a1) circle (\s) node[above left] {$a_1$};
  \draw[fill] (b1) circle (\s) node[above right] {$b_1$};
  \draw[fill] (b2) circle (\s) node[below right] {$b_2$};
  \draw[fill] (a2) circle (\s) node[above right] {$a_2$};
  \draw[fill] (a2') circle (\s) node[below left] {$a_2'$};
  \draw[fill] (b2') circle (\s) node[below left] {$b_2'$};
\end{tikzpicture}

    \caption{
      The arrangement of points in the proof of
      Lemma~\ref{lem:nesting_shadows}, with geodesics labeled by the time
      parameters $a_i,b_i,a_i',b_i'$.
    }
    \label{fig:nesting_shadows}
  \end{figure}

\medskip

  \noindent \textbf{Claim:} $b_1-D-4\epsilon-4\sigma \leq a_2$. Hence, $a_1 \leq  a_2$. 

  \medskip 
  
  \noindent \emph{Proof of Claim:} Suppose the first assertion fails. Notice that   
  \[
    \ell_{\eta}\Big(\left(\max( a_1, a_2),\min( b_1, b_2)\right)\Big)\subset
    \Nc_{\epsilon+\sigma}(\alpha_1\Cc)\cap
    \Nc_{\epsilon+\sigma}(\alpha_2\Cc).
  \]
  Since $\dist(o,\alpha_1\Cc)\leq \dist(o,\alpha_2\Cc)$,  Proposition~\ref{prop:projection to Morse subset}(2) implies that
  \begin{align*}
    a_1-a_2&=\dist(\ell_\eta(a_1),o)-\dist(\ell_\eta(a_2),o) \\&\leq
    \big( \dist(o,\alpha_1\Cc)+3\epsilon+3\sigma \big)-\big(\dist(o,\alpha_2\Cc)-\epsilon-\sigma\big)\leq 4\epsilon+4\sigma.
  \end{align*}
  Thus, 
  $$
  \max( a_1, a_2) \leq a_2 + 4\epsilon+4\sigma.
  $$
  Furthermore, by assumption, 
  $$
  \min( b_1, b_2) > \min( a_2 + D + 4\epsilon + 4\sigma, a_2 + T) \geq a_2 + D + 4\epsilon + 4\sigma.
  $$
  Then  
  $$
  D \geq \diam 
  \Nc_{\epsilon+\sigma}(\alpha_1\Cc)\cap
  \Nc_{\epsilon+\sigma}(\alpha_2\Cc)
  \geq \min( b_1, b_2)-\max( a_1, a_2) > D
  $$
  and we have a contradiction. Hence, the first assertion is true. For the second, notice that 
  $$
  a_1 \leq b_1 - T  \leq b_1 - T_2 =b_1 - D-4\epsilon-4\sigma \leq a_2. 
  $$
  The claim follows.  \hfill $\blacktriangleleft$

\medskip

  Next, Proposition~\ref{prop:projection to Morse subset}(2) implies that 
  $$
  \ell_\eta(a_2), \ell_\xi(a_2') \in 
  \Bc_{3\epsilon+3\sigma}(\pi_{\alpha_2 \Cc}(o))
  $$
  and so Lemma~\ref{lem:HD_between_geod} implies that 
  $$
  \dist^{\rm Haus}( \ell_\eta([0,a_2]), \ell_\xi( [0,a_2']) \leq 6\epsilon + 6 \sigma. 
    $$
    Thus, there exist $a_1', b_1' \in [0,a_2']$ with 
    $$
    \dist( \ell_\eta(a_1), \ell_\xi(a_1')),  \ \dist( \ell_\eta(b_1-D-4\epsilon-4\sigma), \ell_\xi(b_1')) \leq 6\epsilon + 6 \sigma. 
    $$
    Then 
    $$
    b_1' - a_1' \geq b_1 -a_1-D-16\epsilon - 16\sigma \geq T + \phi(\dist(o, \alpha_1 \Cc)) - D-16 \epsilon - 16 \sigma
    $$
    and 
by     Lemma~\ref{lem:following_flats},
    $$
    \ell_\xi( [a_1', b_1']) \subset \Nc_{6 \epsilon + 6\sigma}(\ell_\eta( [ a_1,b_1])) \subset \Nc_{7\epsilon+7\sigma}(\alpha_1 \Cc). 
    $$
    Thus, 
    \begin{equation*}
      \xi \in \phi
      \Sc_{T- D-16 \epsilon - 16 \sigma, 7\epsilon+7\sigma}(\alpha_1 \Cc).  
      \qedhere
    \end{equation*}
  \end{proof}

For $m>n$ and $\alpha\Cc\in\Ac_n$, let 
\[
  I_{T,\epsilon,m}(\alpha\Cc):=\{{\beta\Cc}\in\Ac_m \mid \phi {\Sc}_{T,\epsilon}(\alpha\Cc)\cap \phi {\Sc}_{T,\epsilon}(\beta\Cc)\neq
  \emptyset \}. 
\]

\begin{lemma} Assume $\epsilon > 0$ and fix $T_1, T_2$ as in Proposition~\ref{prop:disjointness} and  Lemma~\ref{lem:nesting_shadows}. There exists $C_2 > 0$ such that: If $T \geq \max(T_1,T_2)$, $m>n$, and $\alpha\Cc\in\Ac_n$, then
  \[
    \sum_{\beta \Cc\in I_{T,\epsilon,m}(\alpha \Cc)} \mu({\Sc}_{T,\epsilon}(\beta\Cc)) \leq C_2 \mu(\phi{\Sc}_{T,\epsilon}(\alpha\Cc)).
  \]
  \label{lem:quasi_ind_lemma2}
\end{lemma}

\begin{proof}
  By Lemma~\ref{lem:nesting_shadows}, 
  $$
  \Sc_{T,\epsilon}(\beta\Cc) \subseteq \phi {\Sc}_{T-C_1,\epsilon+C_1}(\alpha\Cc)
  $$
  whenever $\beta \Cc \in I_{T,\epsilon,m}(\alpha\Cc)$. By Proposition~\ref{prop:disjointness} the sets $\{ \Sc_{T,\epsilon}(\beta \Cc)  : \beta \in \Ac_m\}$ are disjoint. Hence, 
  \[
   \sum_{\beta\Cc \in I_{T,\epsilon,m}(\alpha \Cc)} \mu({\Sc}_{T,\epsilon}(\beta\Cc))\leq \mu\left(\bigcup_{{\beta\Cc}\in I_{T,\epsilon,m}(\alpha\Cc)} {\Sc}_{T,\epsilon}(\beta\Cc)\right) \leq
    \mu(\phi {\Sc}_{T-C_1,\epsilon+C_1}(\alpha\Cc)).
  \]
 Also,  by the   Subset Shadow Lemma  (Theorem~\ref{thm:shadowing_flats}),
    $$
        \mu(\phi {\Sc}_{T-C_1,\epsilon+C_1}(\alpha\Cc)) \asymp  \mu(\phi{\Sc}_{T,\epsilon}(\alpha\Cc)),
              $$
              which completes the proof. 
\end{proof}

\begin{proof}[Proof of Proposition~\ref{prop:quasi_independence}]
Fix $T \geq \max(T_1,T_2)$ and $m>n \geq 1$. Since $T$ and $\epsilon$ are now fixed, we drop the $T,\epsilon$ subscript. Then
  \begin{align*}
    \mu(U_{m}^\phi\cap U_{n}^\phi) & \leq \sum_{\alpha\Cc\in 
    {\Ac}_n}\sum_{{\beta\Cc}\in I_m(\alpha\Cc)} \mu(\phi {\Sc}(\alpha\Cc)\cap \phi {\Sc}(\beta\Cc)) \\
    & \leq \sum_{\alpha\Cc\in 
    \Ac_n}\sum_{{\beta\Cc}\in I_m(\alpha\Cc)} \mu(\phi {\Sc}(\beta\Cc)).
  \end{align*}
  Then by Lemma~\ref{lem:quasi_ind_lemma1}, Lemma~\ref{lem:quasi_ind_lemma2},
  and then Lemma~\ref{lem:quasi_ind_lemma1} again, 
  \begin{align*}
    \mu(U_{m}^\phi\cap U_{n}^\phi) & \lesssim \sum_{\alpha\Cc\in 
    \Ac_n}\sum_{{\beta\Cc}\in I_m(\alpha\Cc)} \mu( {\Sc}(\beta\Cc)) \mu(U_m^\phi) \lesssim  \sum_{\alpha\Cc\in 
    \Ac_n} \mu(\phi{\Sc}(\alpha\Cc))\mu(U_m^\phi) \\
    & \lesssim \mu(U_{m}^\phi) \mu(U_{n}^\phi). \qedhere
  \end{align*}
\end{proof}

\subsection{Finishing the proof of Theorem~\ref{thm:khinchin_1}} Fix
$\epsilon > 0$, $\phi \colon [0,\infty) \rightarrow [0,\infty)$ with $\phi$ slowly varying, and $T > 0$ large enough to satisfy Propositions~\ref{prop:sum condition in BC} and~\ref{prop:quasi_independence}.

Part (1) of Theorem~\ref{thm:khinchin_1} follows directly from part (1) of
the Borel--Cantelli Lemma (Lemma~\ref{lem:borel_cantelli}) and Proposition~\ref{prop:sum condition in BC}.
  
  Part (2) of Theorem~\ref{thm:khinchin_1} requires more work. Suppose for the rest of the section that $K^\phi = \infty$. 
  The key step in the proof is to construct another slowly varying function    with the additional properties from Lemma~\ref{lem:construction of a
  different SVF} below.

\begin{lemma}\label{lem:construction of a different SVF} There exists $\psi
\colon[0,\infty) \rightarrow [0,\infty)$ such that $\psi$ is slowly varying, $K^{\phi + \psi} = \infty$, and $\Gamma \cdot \Theta^{\phi + \psi}_{T,\epsilon} \subset \Theta^{\phi}_{T,\epsilon}$. \end{lemma} 

Assuming the lemma for a moment, we complete the proof. By
Proposition~\ref{prop:quasi_independence}, applied to $\phi+\psi$ and
 part (2) of the Borel--Cantelli Lemma (Lemma~\ref{lem:borel_cantelli}), we have 
$$
\mu(\Theta^{\phi+\psi}_{T,\epsilon})>0.
$$
Then by ergodicity of the $\Gamma$ action on $(\partial X, \mu)$ (see Theorem~\ref{thm:PS ergodic case}), we have 
$$
1 = \mu( \Gamma \cdot \Theta^{\phi+\psi}_{T,\epsilon}) \leq \mu(\Theta^\phi_{T,\epsilon}).
$$

\begin{proof}[Proof of Lemma~\ref{lem:construction of a different SVF}]
Since $K^\phi=\infty$, we can pick $1=n_1 < n_2 < \cdots$ such that 
$$
\sum_{n=n_j}^{n_{j+1}-1} e^{-(\delta(\Gamma)-\delta(\Gamma_0))\phi
(n)} Q(\phi(n)) > j
$$
for all $j\geq 1$. Define 
$$
\psi(t) = \frac{1}{\delta(\Gamma)-\delta(\Gamma_0)}\log \sqrt{j} \quad \text{if} \quad n_j \leq t < n_{j+1}.
$$
By Proposition~\ref{prop:entropy_gap}, we have $\delta(\Gamma) - \delta(\Gamma_0) > 0$ and so $\psi$ is non-negative. It is also straightforward to confirm that $\psi$ is slowly varying. 

Since $Q$ is a positive-valued polynomial, there exists $\lambda > 0$ such that
$$
\inf_{n \geq 1} \frac{Q\big((\phi+\psi)(n)\big)} {Q(\phi(n)) } \geq \lambda.
$$
Then
$$
\sum_{n=n_j}^{n_{j+1}-1} e^{-(\delta(\Gamma)-\delta(\Gamma_0))(\phi+\psi)
(n)} Q\big((\phi+\psi)(n)\big) > \lambda\sqrt{j},
$$
and hence
$$
K^{\phi+\psi} = \infty. 
$$

To show that $\Gamma \cdot \Theta^{\phi+\psi}_{T,\epsilon}$ is contained in
$\Theta^{\phi}_{T,\epsilon}$, fix $\eta \in
\Theta^{\phi+\psi}_{T,\epsilon}$ and $\gamma \in \Gamma$. Then there exist
$m_j \rightarrow \infty$ and $\alpha_j\Cc \in \Ac_{m_j}$ such that
$$
\eta \in
\Sc_{T+(\phi+\psi)(d(o,\alpha_{j}\Cc)),\epsilon}(\alpha_j\Cc)
$$
for all $j \geq 1$. By Lemma~\ref{lem:HD_between_geod}, 
$$
\dist^{\rm Haus}([o,\gamma \eta), [\gamma o, \gamma \eta)) \leq \dist(o,\gamma o)
$$
and so
$$
\gamma \eta \in \Sc_{T+ (\phi + \psi)(d(o,\alpha_j\Cc)), \epsilon+\dist(o,\gamma o)}(\gamma \alpha_j\Cc)
$$
for all $j \geq 1$.
By Proposition~\ref{prop:close for a very long time means very close for a long time},
there exists $C_1 = C_1(\gamma) > 0$ such that
$$
\gamma \eta \in \Sc_{T+ (\phi + \psi)(d(o,\alpha_j\Cc))-C_1, \epsilon}(\gamma \alpha_j\Cc)
$$
for all $j$ sufficiently large.  Next,
$$
\abs{\dist(o,\alpha_j\Cc)-\dist(o,\gamma \alpha_j\Cc)} \leq \dist(o,\gamma
o). 
$$
Then, since $\phi$ is slowly varying, there exists $C_2 =C_2(\gamma,\phi)>0$ 
such that
$$
\gamma \eta \in \Sc_{T+
\phi(\dist(o,\gamma\alpha_j\Cc))+\psi(\dist(o,\alpha_j\Cc))-C_1-C_2, \epsilon}(\gamma \alpha_j\Cc)
$$
for all $j$ sufficiently large.  Finally, since $\dist(o,\alpha_j\Cc)\to\infty$ and hence 
$\psi(\dist(o,\alpha_j\Cc)) \rightarrow \infty$, we have
$$
\gamma \eta \in \Sc_{T+ \phi(\dist(o,\gamma\alpha_j\Cc)), \epsilon}(\gamma \alpha_j\Cc)
$$
for all $j$ sufficiently large. Thus, $\gamma \eta \in \Theta^{\phi}_{T,\epsilon}$ and the proof is complete. 
\end{proof}

\section{Proof of Theorem~\ref{thm:main}}
\label{s:loglaw}

In this section we prove Theorem~\ref{thm:main}. For the rest of the section let $X$, $\Gamma$, $\Cc$, $\Gamma_0$, $T_0$, and $Q$ be as in
Theorem~\ref{thm:main}.  Fix a basepoint $o \in \Cc$ and let $\mu$ be the unique Patterson--Sullivan measure for $\Gamma$ of dimension $\delta(\Gamma)$, see Theorem~\ref{thm:PS ergodic case}.

For $x \in X$ and $\xi \in \partial X$, let $\ell_{x\xi}\colon[0,\infty)\to X$ denote the geodesic ray based at $x$ and limiting to $\xi$. Then fix $\epsilon > 0$ and let 
\[
  \mathfrak q(x,\xi,t)=\begin{cases}
      0 & \textrm{ if }\ell_{x\xi}(t)\not\in
      \Nc_\epsilon(\alpha\Cc)
      \textrm{ for all }\alpha\in\Gamma\\
      \sup|I|& \begin{array}{l} \text{where $I$ is an interval with
      }\ell_{x\xi}(t)\in \ell_{x\xi}(I)\subset
      \Nc_\epsilon(\alpha\Cc)\\
       \text{for some $\alpha \in \Gamma$.}\end{array}
\end{cases}
\]
Notice that 
$$
  \mathfrak q(x,\xi,t) = \mathfrak p_{\Cc, \epsilon}(\ell_{x\xi}, t),
  $$
where $\mathfrak p_{\Cc, \epsilon}$ is the function in Theorem~\ref{thm:main}.

To prove Theorem~\ref{thm:main}, we first prove the following.

\begin{proposition}
  For $\mu$-a.e. $\xi\in\partial X$,
  \[
    \limsup_{t\to\infty}\frac{\mathfrak q(o,\xi,t)}{\log
    t}=\frac1{\delta(\Gamma)-\delta(\Gamma_0)}.
  \]
  \label{prop:main_equiv}
\end{proposition}

\begin{proof} Fix $T > 0$ satisfying Theorem~\ref{thm:khinchin}. For $\kappa>0$, consider the family of functions
  \[
    \phi_\kappa(t)=\kappa \log(t+1).
  \]
  Then 
  \[
    K^{\phi_\kappa}\asymp \sum_{n\in\mathbb N} n^{-\kappa
    (\delta(\Gamma)-\delta(\Gamma_0))}Q(\log n).
  \]
 Proposition~\ref{prop:entropy_gap} implies that $\delta(\Gamma) -\delta(\Gamma_0)>0$ and hence $K^{\phi_\kappa}$ diverges for
  $\kappa \leq \frac1{\delta(\Gamma)-\delta(\Gamma_0)}$   and converges
  otherwise.
  Then Theorem~\ref{thm:khinchin} implies that 
  $\mu(\Theta^{\phi_{\kappa}}_{T,\epsilon})=0$ for
  $\kappa>\frac1{\delta(\Gamma)-\delta(\Gamma_0)}$ and
  $\mu(\Theta^{\phi_{\kappa}}_{T,\epsilon})=1$ for
  $\kappa=\frac1{\delta(\Gamma)-\delta(\Gamma_0)}$.

Let 
  \[
    \kappa_\ast=\frac1{\delta(\Gamma)-\delta(\Gamma_0)}, \qquad
    \kappa_n=\kappa_\ast+\frac1n.
  \]
Notice that 
  $\cup_{n\in\mathbb N}\Theta^{\phi_{\kappa_n}}$ is an increasing union of measure zero sets inside of 
  $\Theta^{\phi_{\kappa_\ast}}$. 
  
  Fix
  \begin{equation}
    \label{e:apply_khinchin}
    \xi\in
    \Theta^{\phi_{\kappa_\ast}} \smallsetminus \bigcup_{n\in\mathbb N}
    \Theta^{\phi_{\kappa_n}},
  \end{equation}
  which is a full measure set. 
  
  Let 
  $$
  A: = \{ \alpha \Cc : \ell_{o\xi} \cap \Nc_\epsilon(\alpha \Cc) \neq \emptyset\}.
  $$
  Since $\xi \in \Theta^{\phi_{\kappa_\ast}}$, the set $A$ is infinite.
  Furthermore, since only a finite number of distinct translates intersect a fixed compact set, see Proposition~\ref{prop:counting_translates}, we can enumerate $A = \{ \alpha_n \Cc\}$ such that 
    $$
    \dist(o,\alpha_1\Cc) \leq     \dist(o,\alpha_2\Cc) \leq \cdots 
  $$
  and moreover $\lim_{n \rightarrow \infty} \dist(o,\alpha_n \Cc) = \infty$. Next let 
    $$
(a_n, b_n): = \left\{ t \geq 0 : \ell_{o\xi}(t) \in \Nc_\epsilon(\alpha_n \Cc)\right\}
$$
(which is indeed an interval by Lemma~\ref{lem:following_flats}). At this point we have not ruled out $b_n =\infty$ (although one can show that each $b_n$ is finite). 

By  Proposition~\ref{prop:projection to Morse subset}(2), there exists $\sigma \geq 0$ such that
  $$|a_n-\dist(o,\alpha_n\Cc)|\leq 3\epsilon+3\sigma.$$  
  Hence, 
  $$
  \lim_{n \rightarrow \infty} a_n=\infty. 
  $$
   Since $\xi \in \Theta^{\phi_{\kappa_\ast}}$, there exists $n_j \nearrow \infty$ such that $\xi\in\phi_{\kappa_\ast}{\Sc}_{T,\epsilon}(\alpha_{n_j}\Cc)$. Fix $t_j \in (a_{n_j}, b_{n_j}) \cap (a_{n_j}, a_{n_j}+\epsilon)$. Then 
    \begin{align*}
    \limsup_{t\to\infty}\frac{\mathfrak q(o,\xi,t)}{\log
    t} \geq \limsup_{j\to\infty}  \frac{\mathfrak q(o,\xi,t_{j})}
    {\log t_{j} }    \geq\limsup_{j\to\infty}  
  \frac{T+\kappa_\ast\log \dist(o,\alpha_{n_j}\Cc) }
    {\log( \dist(o,\alpha_{n_j}\Cc)+4\epsilon+3\sigma)} \geq \kappa_\ast.
  \end{align*}

To prove the upper bound,  fix $s_k \rightarrow \infty$ such that 
  \[
    \limsup_{t\to\infty}\frac{\mathfrak q(o,\xi,t)}{\log
    t}=
    \limsup_{k\to\infty}
    \frac{\mathfrak
    q(o,\xi,s_k)}{\log s_k}.
  \]
  Since the limit is at least $\kappa_\ast > 0$, by passing to a tail we can assume that $\mathfrak
    q(o,\xi,s_k)>0$ for all $k \geq 1$. Then for each $k$, there exists $m_k\in \Nb$ such that $s_k \in (a_{m_k}, b_{m_k})$ and 
$$
b_{m_k}-a_{m_k}= \mathfrak q(o,\xi,s_k).
$$

      \medskip 
      
      \noindent \textbf{Claim:} $m_k \rightarrow \infty$. 
      
      \medskip 
      
      \noindent \emph{Proof of Claim:} Suppose not.  Then after passing to a subsequence we can suppose $m_k = m_1$ for all $k \geq 1$. Then 
      $$
      s_k \in (a_{m_1}, b_{m_1})
      $$
      for all $k \geq 1$, which implies that $b_{m_1}=\infty$ as $s_k \rightarrow \infty$. Thus, 
            $$
  \ell_\xi \big( (a_{m_1},\infty)  \big)\subset \Nc_\epsilon(\alpha_{m_1}\Cc).
      $$
      Recall that $a_{n_j} \rightarrow \infty$ and 
      $$
      \lim_{j \rightarrow \infty} b_{n_j}-a_{n_j}\geq  \lim_{j \rightarrow \infty}T+\kappa_\ast\log \dist(o,\alpha_{n_j}\Cc)=\infty
      $$
      Thus, we have 
      $$
      \lim_{j \rightarrow \infty} {\rm diam} \left( \Nc_\epsilon(\alpha_{m_1}\Cc) \cap \Nc_{\epsilon}(\alpha_{n_j} \Cc) \right) \geq \lim_{j \rightarrow \infty} {\rm diam} \left( \ell_{o\xi}\big((a_{n_j}, b_{n_j})\big) \right)= \infty.
      $$
      Then Proposition~\ref{prop:diam_flat_intersection} implies that $\alpha_{m_1} \Cc = \alpha_{n_j} \Cc$ for $j$ sufficiently large. However, the $\{\alpha_{n_j}\Cc\}$ are distinct translates and so we have a contradiction. Thus, the claim is true.  \hfill $\blacktriangleleft$

      \medskip
      
     Now  passing to a subsequence, we can assume that the $\{m_k\}$ are all distinct and hence the $\{ \alpha_{m_k}\}$ are all distinct. For each $k$, let $j_k\in\mathbb N$ be the smallest number with
  \[ 
\kappa_{j_k}\log
  \dist(o,\alpha_{m_k}\Cc) \leq \mathfrak q(o,\xi,s_k).
  \]  
  If $j_k<J$ infinitely often, then $\xi\in \Theta^{\phi_{\kappa_J}}$ which contradicts
  Equation~\eqref{e:apply_khinchin}. Thus, $j_k\to\infty$. 
Hence,
  \begin{align*}
    \limsup_{t\to\infty} \frac{\mathfrak q(o,\xi,t)}{\log t} & =   \lim_{k\to\infty} \frac{\mathfrak q(o,\xi,s_k)}{\log s_k} \leq  \lim_{k\to\infty} \frac{\mathfrak q(o,\xi,s_k)}{\log a_{m_k}}\\
    & \leq  \limsup_{k\to\infty}
    \frac{\kappa_{j_{k}-1}\log \dist(o,\alpha_{m_k}\Cc)}{\log(\dist(o,\alpha_{m_k}\Cc)-3\epsilon-3\sigma)}=\kappa_\ast.
    \qedhere
  \end{align*}
\end{proof}

Theorem~\ref{thm:main} is a consequence of Proposition~\ref{prop:main_equiv} and the following. 

\begin{lemma} If $x \in X$ and $\xi \in \partial X$, then 
  \[
    \limsup_{t\to\infty}\frac{\mathfrak q(x,\xi,t)}{\log
    t}=\limsup_{t\to\infty}\frac{\mathfrak q(o,\xi,t)}{\log
    t}.
  \]

\end{lemma}

\begin{proof}
Suppose that $\ell_{o\xi}(t)\in\Nc_\epsilon(\alpha \Cc)$. Let
$$
(a,b) : = \left\{ t \geq 0 : \ell_{o\xi}(t) \in 
\Nc_\epsilon(\alpha \Cc)
\right\}
$$
(which is indeed an interval by Lemma~\ref{lem:following_flats}).  Then by 
Lemma~\ref{lem:HD_between_geod}, 
there exist $a',b'$ such that 
$\dist(\ell_{x\xi}(a'),\ell_{o\xi}(a))\leq \dist(o,x)$  and
$\dist(\ell_{x\xi}(b'),\ell_{o\xi}(b))\leq \dist(o,x)$. Thus,
\begin{equation*}
  b'-a'\geq b-a-2\dist(o,x) 
  \label{e:change_basepoint}
\end{equation*}
and by Lemma~\ref{lem:following_flats}, $\ell_{x\xi}([a',b'])\subset
\Nc_{\epsilon+\dist(o,x)}(\alpha\Cc)$. Now, by 
Proposition~\ref{prop:close for a very long time means very close for a
long time}, 
there
exists a constant $K$ (which only depends on $\Cc$, $\epsilon$, and $\dist(o,x)$) such that 
\[
  \ell_{x\xi}([a'-K,b'+K])\subset \Nc_\epsilon(\alpha\Cc).
\]
Thus, for all such $t$, there exists an $s \in [t-K,t+K]$ so that 
\[
  \mathfrak q(x,\xi,s)
  \geq \mathfrak q(o,\xi,t)-2K-2\dist(o,x).
\]
Hence, 
  \[
    \limsup_{t\to\infty}\frac{\mathfrak q(x,\xi,t)}{\log
    t}\geq \limsup_{t\to\infty}\frac{\mathfrak q(o,\xi,t)}{\log
    t}.
  \]
  The proof of the other inequality is exactly the same. 
\end{proof}

\section{Periodic Morse flats and an example}\label{sec:Morse Flats}

Recall that a \emph{$d$-flat}  $F$ in a $\CAT(0)$-space $X$  is a subset isometric to $\Rb^d$ and given a subgroup $\Gamma < \Isom(X)$, a flat is \emph{$\Gamma$-periodic} if its stabilizer in $\Gamma$ acts cocompactly. 

In this section we consider the case of periodic Morse flats and show that after thickening, they satisfy the hypothesis of Theorem~\ref{thm:main}. 
As a corollary, we will prove Theorem~\ref{thm:nonpositive}.

\begin{proposition}\label{prop:Morse flats} Suppose $X$ is a proper $\CAT(0)$-space, $ \Gamma < \Isom(X)$ is a discrete subgroup, and $F$ is a $\Gamma$-periodic $d$-flat. If $F$ is Morse, then there exist a subgroup $\Gamma_0 < \Gamma$  and a closed $\Gamma_0$-invariant convex subset $\Cc$ such that:

\begin{enumerate}
\item  $F\subset \Cc \subset \Nc_r(F)$ for some $r \geq 0$. 
\item ${\rm Stab}_\Gamma(F)$ is a finite index subgroup of $\Gamma_0$. 
\item $\Gamma_0$ acts cocompactly on $\Cc$. 
\item $\Gamma_0$ is almost malnormal in $\Gamma$.
\item $\Cc$ contains all geodesic lines in $X$ which are parallel to a geodesic line in $\Cc$.
\item    For $o\in X$, there exist $N_0,C_0>0$ such that
$$
C_0^{-1}n^d \leq \#\{ \gamma \in \Gamma_0 : n \leq \dist(o,\gamma o) \leq n
+ N_0\} \leq C_0 n^d.
$$
\end{enumerate}
Moreover, if $F$ contains all geodesic lines in $X$ which are parallel to a geodesic line in $F$, then we can choose $\Cc = F$ and $\Gamma_0 = {\rm Stab}_\Gamma(F)$.
\end{proposition} 

\begin{proof} Let 
$$
\Gamma_0 := {\rm Stab}_\Gamma(\partial F) = \{ g \in \Gamma : g \partial F=\partial F\}
$$
and let $\Cc$ be the closure of the convex hull of   the union
of $F$ with all geodesic lines in $X$ parallel to a geodesic line in $F$. Proposition~\ref{prop:Cond 2 holds after thickening} implies that (1) and (5) are true. 

Since $F$ is a union of geodesic lines, $\Cc$ is also the closure of the convex hull of all geodesic lines in $X$ parallel to a geodesic line in $F$.
Furthermore, a geodesic line $\ell$  in $X$ is parallel to a geodesic line in $F$ if and only if the limit points of $\ell$ are in $\partial F$. Hence, $\Gamma_0$ preserves the set of geodesic lines in $X$ parallel to a geodesic line in $F$, which implies that $\Cc$ is $\Gamma_0$-invariant. 

Since ${\rm Stab}_\Gamma(F)$ acts cocompactly on $F$, (1) implies that ${\rm Stab}_\Gamma(F)$ acts cocompactly on $\Cc$. Then, since ${\rm Stab}_\Gamma(F) < \Gamma_0$, (3) is true. Since ${\rm Stab}_\Gamma(F)$ and $\Gamma_0$ both act cocompactly on $\Cc$, (2) is true. Then (6)  is a consequence of (2) and the fact that $F$ is isometric to $\Rb^d$.

It remains to prove (4). 
  Suppose $\Gamma_0$ is not almost
malnormal in $\Gamma$. 
Then there exists some $g \in \Gamma -
\Gamma_0$ where $\Gamma_0 \cap g\Gamma_0g^{-1}$ is infinite. It follows
from the Bieberbach theorem (see ~\cite[Theorem 3.2.1]{Wolf}) that ${\rm Stab}_\Gamma(F)$ contains a finite
index subgroup $H$ where every element acts by translations on $F$. By (2),
$H$ has finite index in $\Gamma_0$ and so there exists a non-identity
element $h \in H \cap \Gamma_0 \cap g\Gamma_0g^{-1}$. Then $h$ translates
a geodesic line $\ell_1$ in $F$, i.e. there exists $T>0$ such that for any unit speed parametrization $\ell_1 : \Rb \rightarrow X$ we have $h\ell_1(t) = \ell_1(t \pm T)$.  Furthermore, $gHg^{-1}$ has finite index in $g\Gamma_0g^{-1}$ and acts by
translations on $gF$. Then by replacing $h$ by a power, we can also assume
that $h \in gHg^{-1}$ and hence translates a geodesic line $\ell_2$ in $gF$. Since $\ell_1$ and $\ell_2$ are both 
translated by $h$, they are parallel (see ~\cite[Part II, Theorem 6.8]{BH}).

Since $gF$ is a union of geodesic lines parallel to $\ell_2$, (5) implies
that  $gF \subset \Cc$. Since $\Cc \subset \Nc_r(F)$, we then have
$g\partial F \subset \partial F$. Since $g\partial F$ and $\partial F$ are
both homeomorphic to the sphere of dimension $d-1$, the invariance of
domain theorem implies that $g \partial F$ is open in $\partial F$. Since
$g \partial F$ is also closed and connected, we have $g \partial F =
\partial F$. Thus, $g \in \Gamma_0$ and we have a contradiction.
\end{proof} 

\begin{example}\label{ex:isolated flats} In~\cite{HK2005}, Hruska--Kleiner defined $\CAT(0)$ spaces with isolated flats. By ~\cite[Theorems 1.2.1 and 1.2.3]{HK2005}, ~\cite[Theorem 2.13]{Sisto2013}, and Theorem~\ref{thm:equiv of Morse} the flats appearing in their definition are periodic and Morse. Hence, by Proposition~\ref{prop:Morse flats}, thickenings of them satisfy Theorem~\ref{thm:main}.

\end{example}

\subsection{Proof of Theorem~\ref{thm:nonpositive}}\label{sec:proof of thm nonpositive}

As in the introduction, let $M$ be an oriented closed non-geometric irreducible 3-manifold 
obtained by gluing hyperbolic pieces along tori. By a result of Leeb~\cite{Leeb1995}, we can endow
$M$ with a non-positively curved Riemannian metric. Furthermore, using the flat torus theorem~\cite{GW1971,LY1972}, we can assume the gluing tori are flat and totally geodesic.

Let $\tilde M$ denote the universal cover of $M$ and let $\Gamma : = \pi_1(M)< \Isom(\tilde M)$. Manning~\cite{Manning} proved that the topological entropy of the geodesic flow on $M$ coincides with the critical exponent of $\Gamma$, i.e.
\begin{equation}\label{eqn:Manning}
h_{top}(M) = \delta(\Gamma). 
\end{equation} 
Since $M$ has rank one, there exists a unique Patterson--Sullivan measure
$\mu$ for $\Gamma$ of dimension $\delta(\Gamma)$ \cite{Knieper1997}.

  We first verify that the assumptions of Theorem~\ref{thm:main} hold.
  Since $M$ is compact,
  \[
    \#\{\gamma\in\Gamma \mid d(o,\gamma o)\leq t\} \asymp \vol (\mathcal
    B_t(o))
  \]
  where $\mathcal B_t(o)$ is the metric ball of radius $t$ about $o$ in
  $\tilde M$. Then Assumption~\ref{item:Gamma_counting_assumption} follows
  by integrating the statement in \cite[Theorem 5.1]{Knieper1997}.

Let $\Fc$ denote the set of 2-flats in $\tilde M$ which project to gluing tori in the geometric decomposition. Since the tori in $M$ are disjoint,  $\tilde M$ has  isolated flats in the sense of Hruska--Kleiner~\cite{HK2005}. Thus, each flat in $\Fc$ is Morse (see Example~\ref{ex:isolated flats}) and by definition  periodic. 

Now fix a gluing torus $\mathbb{T}$ and a flat $F \in \Fc$ covering
$\mathbb{T}$. As in the statement of Theorem~\ref{thm:nonpositive}, let
$\hat{\mathbb{T}}$ denote the union of all flat tori parallel to
$\mathbb{T}$. As in the proof of Proposition~\ref{prop:Morse flats}, let
$\Cc \subset \tilde M$ denote the closure of the convex hull of $F$ and all
geodesic lines in $\tilde M$ parallel to a geodesic line in $F$. 

By Proposition~\ref{prop:Morse flats}, there is a subgroup $\Gamma_0$ acting
cocompactly on this set $\Cc$, both of which 
satisfy the hypotheses of Theorem~\ref{thm:main}. Hence,
for $\mu$-almost every $\xi \in \partial X$ and
every unit speed geodesic ray $\ell \colon [0,\infty) \rightarrow \tilde M$ limiting to $\xi$, we have 
\begin{equation}\label{eqn:asym for wrong function}
\limsup_{t \rightarrow \infty} \frac{\mathfrak p_{ \Cc,\epsilon}(\ell, t)}{\log t} = \frac{1}{\delta(\Gamma)}.
\end{equation}

\begin{lemma} $\Cc$ projects to $\hat{\mathbb{T}}$. \end{lemma}

\begin{proof} Let $\hat{F}$ denote the lift of $\hat{\mathbb{T}}$ containing $F$. Notice that $\hat F$ coincides with the union of flats parallel to $F$ and is isometric to a closed interval times $F$. Hence, $\hat F \subset \Cc$. 

If $x \in \Cc \smallsetminus F$, then there exists a geodesic line $\ell :
\Rb \rightarrow \tilde M$ containing $x$ which is parallel to a geodesic
line $\ell_0 : \Rb \rightarrow \tilde M$ in $F$. The union of all geodesic
lines parallel to $\ell_0$ is isometric to $C \times \Rb$ where $C$ is
closed and convex and each fiber $\{c\} \times \Rb$ corresponds to a
geodesic~\cite[Part II, Chapter 2, Theorem 2.14]{BH}. Since this set
contains $F \cup \ell_0$ and $\tilde M$ is Riemannian, we must have that
$C$ is isometric to $\Rb$ times a closed interval. It follows that 
$\ell$ is contained in a flat parallel to $F$ and hence  $x \in \hat F$. 
\end{proof}

Let $\mathfrak p_{\epsilon} : T^1 M \times [0,\infty) \rightarrow [0,\infty)$ denote the function in Theorem~\ref{thm:nonpositive} and let $\pi : \tilde M \rightarrow M$ denote the projection. Then, given  a unit speed geodesic ray $\ell \colon
[0,\infty) \rightarrow X$ we have 
\begin{itemize}
\item $\mathfrak p_{\epsilon}(d(\pi)\ell'(0), t) = 0$ if  $\ell(t)$ is not in $\Nc_\epsilon(\cup_{\alpha \in \Gamma} \alpha\Cc)$.
\item Otherwise, $\mathfrak p_{\epsilon}(d(\pi)\ell'(0),t)$ is the size of the maximal interval
$I$ containing $t$ such that $\ell(I) \subset \Nc_\epsilon(\cup_{\alpha \in \Gamma} \alpha\Cc)$.
\end{itemize}
Recall that $\mathfrak p_{\Cc, \epsilon}(\ell,t)$ records the size of the maximal interval contained in the $\epsilon$-neighborhood of a \textbf{single} $\Cc$ translate. Hence, $\mathfrak p_{\epsilon}(d(\pi)\ell'(0), t) \geq \mathfrak p_{\Cc, \epsilon}(\ell,t)$

\begin{lemma}\label{lem:penetration functions are the same} For $\epsilon > 0$ sufficiently small: If  $\ell \colon [0,\infty) \rightarrow \tilde M$ is a unit speed geodesic and $\pi : \tilde M \rightarrow M$ is the projection, then 
$$
\mathfrak p_{\epsilon}(d(\pi)\ell'(0), \cdot) = \mathfrak p_{ \Cc,\epsilon}(\ell, \cdot).
$$
\end{lemma} 

\begin{proof} Since $\hat{\mathbb{T}}$ is embedded in $M$, for $\epsilon > 0$ sufficiently small disjoint translates of $\Cc$ have disjoint $\epsilon$-neighborhoods. For such $\epsilon > 0$ we have the desired equality. \end{proof} 

Now Lemma~\ref{lem:penetration functions are the same}, Equation~\eqref{eqn:Manning}, and Equation~\eqref{eqn:asym for wrong function} imply that for $\mu$-almost every $\xi \in \partial \tilde M$ and
every unit speed geodesic ray $\ell \colon [0,\infty) \rightarrow \tilde M$ limiting to $\xi$, we have 
\begin{equation*}
\limsup_{t \rightarrow \infty} \frac{\mathfrak p_{ \epsilon}(d(\pi) \ell'(0), t)}{\log t} = \frac{1}{h_{top}(M)}.
\end{equation*}
Thus, Theorem~\ref{thm:nonpositive} follows from the explicit definition of the measure of maximal entropy in~\cite[pg. 298]{Knieper1998}.

\bibliographystyle{alpha}
\bibliography{geom}

\end{document}